\documentclass[12pt]{amsart}

\usepackage{geometry,enumitem}

\usepackage{graphicx}
\usepackage{color}
\usepackage{subcaption}

\usepackage{cite}
\usepackage{hyperref}

\numberwithin{equation}{section}
\newtheorem{thm}{Theorem}[section]
\newtheorem{theorem}[thm]{Theorem}

\newtheorem{proposition}[thm]{Proposition}
\newtheorem{lemma}[thm]{Lemma}
\newtheorem{remark}[thm]{Remark}

\newtheorem*{proposition*}{Proposition}
\newtheorem*{lemma*}{Lemma}
\newtheorem*{theorem*}{Theorem}
\newtheorem*{corollary*}{Corollary}
\newtheorem*{conjecture*}{Conjecture}

\newcommand{\nn}{\nonumber}
\newcommand{\R}{\mathbb{R}}

\newcommand{\mK}{\ensuremath{\mathcal{K}}}

\title{Optimal Transport with Controlled Dynamics and Free End Times}

\author{Nassif Ghoussoub}
\author{Young-Heon Kim}
\author{Aaron Zeff Palmer}
\address{Department of Mathematics, University of British Columbia, Vancouver BC Canada V6T 1Z2}{
\email{ nassif@math.ubc.ca,  yhkim@math.ubc.ca, azp@math.ubc.ca}
\thanks{
The three
authors are partially supported by  the 
Natural Sciences and Engineering Research Council of Canada (NSERC)
}

\date{\today}

\begin{document}

\begin{abstract}
	We consider optimal transport problems where the cost is optimized over controlled dynamics and the end time is free. Unlike the classical setting, the search for optimal transport plans also requires the identification of optimal ``stopping plans," and the corresponding Monge-Kantorovich duality involves the resolution of a Hamilton-Jacobi-Bellman quasi-variational inequality.  We discuss both Lagrangian and Eulerian formulations of the problem, and its natural connection to Pontryagin's maximum principle. We also exhibit a purely dynamic situation, where the optimal stopping plan is a hitting time of a barrier given by the free boundary problem associated to the dual variational inequality. This problem was motivated by its stochastic counterpart, which will be studied in a companion paper.  	
\end{abstract}

\maketitle
 
\tableofcontents

%\marginpar{We may want to reduce the length of the introduction...YH.}
\section{Introduction}\label{sec:intro}	
	With a purpose of connecting mass transport with Mather theory, Bernard and Buffoni \cite{B-B} considered  ``generating functions"  on a compact manifold $M$ as cost functionals. These are    
	%fixed end-points dynamic problems 
	of the following type:
	\begin{equation*}%\label{eqn:Bernard_Buffoni}
		c_T(x,y):=\inf\Big\{\int_0^TL\big(t, \gamma(t), {\dot \gamma}(t)\big)\, dt; \gamma\in C^2([0, T], M);  \gamma(0)=x, \gamma(T)=y\Big\},
	\end{equation*}
	where $[0, T]$ is a fixed time interval, and $L: \R^+\times TM \to \R$ is a given Lagrangian that is convex in the second variable of the tangent bundle $TM$. The corresponding optimal transport problem consists of minimizing the total cost of {\it ``transport plans"} between two given probability distributions $\mu$ and $\nu$ on $M$, 
	\begin{equation*}
		V_{c_T}(\mu, \nu):=	\inf\Big\{\int_{M\times M} c_T(x, y) \, d\pi; \pi\in \Pi(\mu,\nu)\Big\}.
		%=\sup\Big\{\int_{M}\phi_T(y)\, d\nu-\int_{M}\phi_0(x)\, d\mu;\,  \phi_T, \phi_0 \in \mK(c_T)\Big\}.
	\end{equation*}
	Here $\Pi(\mu,\nu)$ is the set of {\it transport plans from $\mu$ to $\nu$}, that is those probability measures $\pi$ on $M\times M$ whose first (resp., second) marginal is $\mu$ (resp., $\nu$).
	 
	The consideration of such cost functionals showcased the fact that the theory of optimal transport can be seen as a natural extension of certain aspects of classical mechanics, where one considers an ensemble of identical, non-interacting `particles', which are transported from a given source distribution $\mu$, to a desired target distribution $\nu$. 

	% Fathi and Figalli \cite{F-F} eventually dealt with the case where $M$ is a non-compact Finsler manifold, while Agrachev and Lee \cite{A-L} considered the case of sub-Riemannian manifolds. 
	Note that standard cost functionals of the form  $c(x, y)=g(|y-x|)$, where $g$ is convex, are particular cases of the dynamic formulation, since they correspond to Lagrangians of the form $L(t, x, v)=g(v)$. %g(Tv)/T$. 
	The original Monge problem dealt with the cost $c(x, y)=|y-x|$ (\cite{M}, \cite{S}, \cite{E-G}, \cite{V1}, \cite{V2}) and was constrained to those probabilities in ${\Pi}(\mu, \nu)$ that are supported by graphs of measurable maps from $X$ to $Y$ pushing $\mu$ onto $\nu$.  Brenier \cite{B1} considered the important quadratic case  $c(x,y)=|y-x|^2$. This was followed by a large number of results addressing costs of the form $g(|y-x|)$, where $g$ is either a convex or a concave function \cite{G-M}.
	
	In this paper, we develop mass transport theory in two new directions: For one, we shall formulate it as a natural extension of optimal control theory, the latter being an extension of the  classical calculus of variations to `nonsmooth' or `semi-discrete' problems that arise frequently in engineering \cite{P}.  Instead of optimizing value functionals when the Lagrangian is a function of velocity, the controlled value might only be optimized over a set of controls, $A(t)\in\mathbb{A}$, that prescribe the velocity. The `state' variable can then obey dynamics of the form
	\begin{align}% \label{eqn:dynamics}
		\dot{\gamma}(t)=f\big(\gamma(t),A(t)\big).\nn
	\end{align}
	These type of  problems with a fixed end time, that is, when all trajectories stop at a given time $T$, has been studied by Agrachev and Lee \cite{A-L}. We shall consider the case when the end time is free, that is, when not given a priori, and is motivated by Skorokhod type embedding problems that we study in \cite{GKPS} and the stochastic control counterpart that we address in \cite{GKP2}.	Examples of such control problems arise in engineering problems such as allocation of emergency resources \cite{Fi}.  \\
	The new elements of our approach include:
	\begin{itemize}
		\item The use of Kantorovich type duality to define an obstacle problem associated to a corresponding Hamilton-Jacobi equation.  The relevance of this approach will be even more apparent in our companion paper \cite{GKPS}, where it is used to construct Skorokhod embeddings as hitting times of a corresponding free boundary problem. This approach has  already been used in particular one-dimensional Skorokhod problems by Cox and collaborators \cite{cox2013root}, \cite{beiglboeck2017optimal}. 

		\item An equivalent Eulerian formulation to the optimal transportation problem, with which we can again employ PDE methods.

%		\item Developing conditions to imply the free boundary to the obstacle problem uniquely determines the optimal end time.

	\end{itemize}

	The transport associated to control problems with free end times consists of  considering cost functionals of the following type:
	\begin{align}%\label{eqn:MKCost0}
		c(x,y):=\inf_{\tau,A(\cdot)}\Big\{\int_0^{\tau} K\big(t,\gamma(t),A(t)\big)dt;\ \dot{\gamma}(t)=f\big(\gamma(t),A(t)\big),\ \gamma(0)=x,\ \gamma(\tau)=y\Big\},\nn
	\end{align}	
	and the corresponding optimal total cost $V_c(\mu,\nu)$.  To our knowledge (and surprise), this variant of optimal control theory and optimal transportation has not been studied in the mathematical literature, despite the elegant mathematics that it yields, which we present in this paper. 

	Our approach follows the, by now standard, pattern of  identifying the dual of our optimal transportation problem, and use the properties of its extremals to infer properties of the primal problem. Recall that under very general conditions on a cost functional $c$, the Monge-Kantorovich duality yields that	 
	\begin{equation*}
		V_c(\mu, \nu)%:=	\inf\Big\{\int_{M\times M} c_T(x, y) \, d\pi; \pi\in \Pi(\mu,\nu)\Big\}.
		=D_c(\mu,\nu):=\sup\Big\{\int_{M}\psi(y)\, d\nu-\int_{M}\phi(x)\, d\mu;\,  \psi, \phi \in \mK(c)\Big\},
	\end{equation*}
	where $\mK(c)$ is the set of functions $\psi\in L^1(M, \nu)$ and $\phi\in L^1(M, \mu)$ such that 
	$$
		\psi(y)-\phi(x) \leq c(x,y) \quad \hbox{ for  all $(x,y)\in M\times M.$}
	$$
	The pairs of functions in $\mK(c)$ can be assumed to satisfy 
	\[
		\psi(y)=\inf_{x\in M} c(x, y)+\phi(x) \quad {\rm and} \quad\phi(x)=\sup_{y\in M} \psi(y)-c(x, y).
	\]
	These will be called {\it optimized Kantorovich potentials}, and for reasons that will become clear later, we shall say that $\phi$ (resp., $\psi$) is an initial (resp., final) Kantorovich potential.

	A crucial contribution by Bernard-Buffoni was to relate optimized Kantorovich potentials to solutions of corresponding  Hamilton-Jacobi equation, 
	\begin{align}\nonumber
			\frac{\partial}{\partial t}J+H(t, x, \nabla J)=&\ 0 \,\, {\rm on}\,\,  [0, T]\times M,
	\end{align}
	with either the forward or the  backward conditions, 
 	\begin{align}%\label{HJ+} 
		\hfill J(0, \cdot )=&\ \phi (\cdot ),\nn\\
		%\label{HJ-} 
		\hfill J(T, \cdot)=&\ \psi (\cdot) \,\, {\rm on}\,\, M, \nonumber
 	\end{align}
	where the Hamiltonian on $[0, T] \times T^*M$ is defined by
	$$
	H(t, q, p):=\sup_{v\in T_qM}\{\langle v, p\rangle -L(t, q, v)\}.
	$$
	This connection, coupled with the fact that the optimal transport map is often given by the Hamiltonian flow, clearly illustrates our claim above on how the theory of mass transport is a natural extension of the calculus of variations view of classical mechanics.
	The Eulerian formulation of mass transport, as initiated by Brenier-Benamou \cite{Be-Br} in the case of a quadratic cost, also leads to its identification as a version of the calculus of variation but on Wasserstein space of probability measures.  

	Following L. Pontryagin, we define the Hamiltonian in terms of the generalized momentum, or costate, $p$ as follows:
	\begin{align}\label{eqn:Hamiltonian-intro}
		H(t,q,p)=\sup_{A\in \mathbb{A}} \Big\{p\cdot f(q,A) -K(t,q,A)\Big\}.
	\end{align}
	Now, under certain natural hypothesis on $K$ and $f$, the cost $c$ becomes lower semi-continuous and Lipschitz in the second variable, and it is then standard (c.f. \cite{V1}) to conclude that the Kantorovich duality holds with attainment of both primal and dual problems. In our particular case, we show in Section~\ref{sec:Monge-Kantorovich}, the following duality
	\begin{align}\label{eqn:dual_problem-intro}
		V(\mu,\nu)=D(\mu, \nu):=\sup\Big\{ \int_{\R^n} \psi\ d\nu-\int_{\R^n} J_\psi(0,\cdot) d\mu ; \, \psi\in C(\R^n)\Big\},  
	\end{align}
	where  $J_\psi(t,\cdot)$ is the `value' or cost-to-go function, which represents maximum payoff when the terminal-payoff is capped at $\psi$, that is 
	\begin{equation}\label{eqn:value}
	 	J_\psi(t,q)=\sup_{\tau,A(\cdot)} \left\{ \psi\big(\gamma(\tau)\big) - \int_t^\tau K(s,\gamma(s),A(s)\big)ds;\ \dot{\gamma}(s)=f\big(\gamma(s),A(s)\big),\ \gamma(t)=q\right\}.%,
	\end{equation}
	The dynamic programming principle then equates $J_\psi$ with viscosity solutions of the Hamilton-Jacobi-Bellman quasi-variational inequality \cite{bardi2008optimal}, 
	\begin{align}\label{eqn:informalHJB}
		\frac{\partial }{\partial t}J+H\big(t,q,\nabla J\big)\leq&\ 0 \ {\rm on}\ \R^+\times \R^n\\
			\psi - J\leq&\ 0\ {\rm on}\ \R^+\times \R^n\nn,
	\end{align}
	where one of the two alternatives holds with equality --restated as (\ref{eqn:HJB}) below.   This is a key ingredient in our analysis. The function $J_\psi$ represents the  backward propagation of the cost such that, if differentiable, $-\nabla J_\psi(0,x)$ is the derivative of the cost of an optimal trajectory with respect to the starting position $x$.  
	While for each $(t,q)$, we have $J_\psi(t,q)\geq \psi(q)$, (\ref{eqn:informalHJB}) also implies that $J_\psi(0,\gamma(0))\leq \psi (\gamma(\tau))$ for any  maximizing trajectory of \eqref{eqn:value}.  These competing inequalities determine the end time $\tau$ to lie in the free boundary of the Hamilton-Jacobi inequality, i.e.\ the boundary of the set where $J_\psi = \psi$. 
		
	We also give an Eulerian formulation of the primal problem $V(\mu, \nu)$ as follows. We consider \emph{density processes} $\rho:\R^+\rightarrow \mathcal{M^+}(\R^{n}\times \mathbb{A})$ and  \emph{stopping distributions} $\tilde{\rho}\in \mathcal{M^+}(\R^+\times \R^n)$, where $\mathcal{M}^+$ represent the cone of non-negative Radon measures on the given metric space.  We then introduce ${\mathcal C}(\mu, \nu)$ to be the class of all {\it admissible process densities} $\rho$, starting at $\mu$, i.e., $d\mu (x) = \int_{\mathbb{A}} d\rho(0, x, \cdot ) $,  with a stopping distribution $\tilde \rho$, i.e.,  $d\nu (x) = \int_{\R^+} d\tilde{\rho} (\cdot, x)$, and satisfying (weakly) the following continuity equation: 
	\begin{align}\label{eqn:BB-continuity}
 		\tilde{\rho}+\frac{\partial }{\partial t}\left(\int_{\mathbb{A}}d\rho\right) + {\rm div}\Big(\int_{\mathbb{A}}f d\rho\Big)=0.
	\end{align}
	We then show  in Theorem~\ref{thm:Eulerian_equality} the following analogue of a celebrated result of Benamou-Brenier\cite{Be-Br}:
	\begin{align}\label{eqn:BB}
		V(\mu,\nu)=W(\mu,\nu):=\inf\Big\{\int_{\R^+}\int_{\mathbb{A}}   \int_{\R^{n}} K(t,q,A)\rho(t,dq,dA)\ dt; \, \rho \in {\mathcal C}(\mu, \nu) \Big\}.
	\end{align}
	We emphasize that unlike the classical case, where mass is conserved while moving in time and then stops at once at the given final time, our admissible process densities $\rho$ drop mass over time before ending with a stopping distribution $\tilde \rho$.  

	We next consider necessary optimality conditions for when the cost $c(x,y)$ of moving  state $x$ to $y$ is attained in a --not necessarily unique-- time $\tau^{x,y}$,  via a control policy $A^{x,y}$. Under suitable conditions, which %that
	ensure that the primal problem is attained at a probability measure $\pi$, and the dual problem is attained at some function $\psi$ with  $J_\psi$ as the corresponding solution of the Hamiliton-Jacobi inequality \eqref{eqn:informalHJB}, we show in Theorem~\ref{thm:Pontryagin} that 
	$$
		J_\psi(\tau^{x,y},y)=\psi(y) \quad \hbox{for $\pi$-a.e.\ $(x,y)$.}
	$$
	And just like the fixed-end case, but this time resorting to Pontryagin's maximum principle (Theorem~\ref{thm:Pontryagin}, item (2)), if an optimal path exists for $(x, y)$ on the support of $\pi$, then it is given by the projection $\gamma(t)$ of the Hamiltonian flow $(\gamma(t), p(t))$, where $\gamma (t)$ starts at $x$ and ends at $y$, while the costate $p(t)$ starts at  $\nabla J_\psi(0, x)$ and ends at $\nabla \psi(y)$. Moreover, 
	\begin{enumerate}
		\item	for a.e.\ $t<\tau^{x,y}$, the control is one where the supremum of the Hamiltonian 
 				$
					H\big(t,\gamma(t),p(t)\big)=p(t)\cdot f\big(\gamma(t),A(t)\big)-K(t,\gamma(t),A(t)\big)
				$ is attained,
		\item while at the end time $\tau=\tau^{x,y}$, the \emph{transversality} condition  $H\big(\tau,\gamma(\tau),p(\tau)\big)=0$ holds.
	\end{enumerate}
	Note that the existence of an optimal transport plan $\pi$ does not necessarily imply that for $\pi$-a.e. $(x,y)$, the cost $c(x,y)$ defined in \eqref{eqn:MKCost} is attained by some optimal pair $\left(\tau, A\right)$. Indeed, 
	without the concavity of $p \cdot f(q,A)-K(t,q,A)$  in \eqref{eqn:Hamiltonian-intro},  we cannot expect attainment of the optimal control policy by a map $t\mapsto A(t)$ for each initial position $x$ and end position $y$. However, one can go around this problem by relaxing the problem and considering 
	\emph{Young probability measures} $t\mapsto \alpha_t\in \mathcal{P}(\mathbb{A})$. But these are  already subsumed under the \emph{admissible process densities} introduced for the Eulerian formulation  \eqref{eqn:BB-continuity}.
	With a suitable coercivity condition for $t$ and $A$, we shall show attainment in the Eulerian formulation at an optimal pair $(\rho, \tilde{\rho})$ (see Theorem~\ref{thm:Eulerian_verification}). There, we also prove that the support of the optimal stopping distribution $\tilde{\rho}$ is contained in the set where $J_\psi=\psi$.
	
	It is important to note that the end time $\tau$ may not be determined uniquely for each given initial point $x$. In general, it is given by a distribution over $\mathbb{R}^+$, that we shall call  {\it stopping plan}, by analogy with the randomized stopping times arising in the theory of stochastic processes, and which will be even more apparant in our companion paper \cite{GKP2}. It is a natural question to see when such stopping plan, when optimized,  is actually a stopping time, i.e., given by a uniquely determined end time (still depending on the initial point $x$). In Section \ref{sec:hitting_times},
	we shall isolate conditions on the running cost (now considered as a Lagrangian) that allow for the determination of an optimal end time,  given by a stopping plan, to be the hitting time $\tau$ of a graph $(y,s({y}))$ of a function $s$ determined by an optimal dual function $\psi$ and its associated $J_\psi$.   As we show in  %Proposition~\ref{prop:dualStructure}.
	Theorem~\ref{thm:Monge_map}, this turned out to be satisfied under the conditions that $t\to K(t,q,A)$ is either strictly increasing or strictly decreasing in time. In either case,   the graph $(y, s(y))$ is, as expected, the free boundary of the variational inequality (\ref{eqn:informalHJB}). A similar approach to proving monotonicity of the solution $J_\psi$ to the variational inequality was taken by one of the authors and A.\ Vladimirsky for an optimal stopping problem with probabilistic constraints in a discrete setting \cite{P-V}.

	In Section~\ref{sec:examples}, we provide a few examples mostly focusing on non-smooth one dimensional cases, while in the final section, we will compare our setting with the classical optimal transport problem, see e.g. \cite{G-M}. 
	Indeed,  if we set the control set to be $\mathbb{A}=S^{n-1}\subset \R^n$ and $f(q,A)=A$, we can view the control problem as minimizing over trajectories with velocity bounded by 1. Consider now the cost is given only by a time penalty that is the derivative of a function $g$ with $g(0)=0$,
	$
		K(t,q,A)=g'(t).
	$
	The optimal trajectory connecting $x$ and $y$ is then the straight line that reaches $y$ at time $\tau^{x,y}=|y-x|$.  Thus the cost associated to travelling from $x$ to $y$ in the optimal amount of time is therefore 
	$
		c(x,y)=g(|y-x|).
	$
	The cases, considered above, when the running time increases (resp., decreases) correspond to when when $g$ is convex (resp., concave).

\section{The primal problem and its dual  \`a la Monge-Kantorovich} 
	\label{sec:Monge-Kantorovich}
	In this section, we prove the no duality gap equation \eqref{eqn:dual_problem-intro}, namely, that $V(\mu, \nu) = D(\mu, \nu),$ and that both optimization problems 
	are attained under suitable conditions on the running cost $K$,  the function $f$, and the source and target measures $\mu$ and $\nu$, such as the following.
	\begin{enumerate}[label=\textbf{H\arabic*}] \setcounter{enumi}{-1}
		\item \label{itm:continuity} 
			The control set $\mathbb{A}$ is a complete seperable metric space, and the running cost $K(t,q,A)$ is non-negative and continuous in $t$, $q$ and $A$.

		\item \label{itm:controllable} 
			The velocity function $f$ is Lipschitz with respect to $q$ and continuous with respect to $A$, and the span of $\big\{f(q,A);\ A\in \mathbb{A}\big\}$ is $\R^n$ for each $q$.

		\item \label{itm:compact} 
			The distributions $\mu$ and $\nu$ have compact support in $\R^n$.
	\end{enumerate}

	\begin{remark}
		These assumptions can be weakened,  but doing so incurs technical difficulties and the weakening of some results. In particular, $f$ could depend also on time for our results through Section {\normalfont\ref{sec:optimality_conditions}}, but we require that $f$ is independent of time for the monotonicity properties of Section {\normalfont\ref{sec:hitting_times}}.  In many models {\normalfont \ref{itm:controllable}} will not be satisfied, in which case the feasibility of the primal problem (i.e.\ $V(\mu,\nu)<\infty$) may be difficult to determine, and it is unclear in general if the dual attainment of Section {\normalfont\ref{sec:Eulerian}} holds.
		The assumption {\normalfont \ref{itm:compact}} is mainly for simplicity, and one may relax it to certain decay conditions of $\mu$ and $\nu$ as $x\to \infty$. 
	\end{remark}

	Our primal optimal transport problem is to minimize the total transportation cost of a transportation plan, $\pi \in \Pi(\mu,\nu)$, with its marginals prescribed as the initial and target distribution:
	\begin{align}\label{eqn:transport_problem}
		V(\mu,\nu) := \inf_{\pi\in \Pi(\mu,\nu)}\Big\{ \int_{\R^n}\int_{\R^n}c(x,y)\pi(dx,dy)\Big\}.
	\end{align}
	
	\begin{remark} We shall use  the following definition of the transportation cost:
		\begin{align}\label{eqn:MKCost}
			c(x,y)=\inf_{\tau,A(\cdot)}\Big\{\int_0^{\tau} K(t,\gamma(t),A(t)\big)dt;\ \gamma(t)= \int_0^tf\big(\gamma(s),A(s)\big)ds+x,\ \gamma(\tau)=y\Big\}, 
		\end{align}	
		which is slightly different from the strong form of the dynamics mentioned in the introduction. 
		This allows for non-differentiable paths, in particular if $A(\cdot)$ is not continuous, such as when it is alternating between two discrete values $\pm 1$. 
	\end{remark}

	\begin{lemma}\label{lem:lsc}
		Given hypotheses {\normalfont \ref{itm:continuity}} and {\normalfont \ref{itm:controllable}}, the optimal transportation cost $c$ of {\normalfont (\ref{eqn:MKCost})} is finite-valued and   continuous. Moreover $y \mapsto c(\cdot,y)$ is Lipschitz continuous.  
	\end{lemma}
	\begin{proof}
		For a compact set  $Z \subset \R^n$, \ref{itm:controllable} implies that for any $x$ and $y$ in $Z$ there is a controlled trajectory with $\gamma(0)=x$, $\gamma(\tau)=y$, and $\tau \leq C_1 |x-y|$.  Hypothesis \ref{itm:continuity} together with the compactness of $Z$, imply that the cost is bounded by $C_2\tau$.  
		Thus $c(x,y)\leq \Lambda_Z|x-y|<\infty$.  
		
		We next consider a second pair $x'$ and $y'$ in $Z$, and let $(\tau,A)$ be within $\epsilon$ of attaining the cost $c(x',y)$ and consider a controlled-trajectory connecting $y$ and $y'$ in time $\tau_1\leq C_1|y'-y|$.  Then, the concatenated path gives
		\begin{align}\label{eqn:Lip-c}
 			c(x',y')\leq c(x',y)+\epsilon +\Lambda_Z|y-y'|.
		\end{align}	
		In the limit as $x'\rightarrow x$ and $y'\rightarrow y$ the trajectories converge to a trajectory transporting $x$ to $y$ by continuity of the solution to an ODE with respect to the initial condition, thanks to assumption \ref{itm:controllable}.  Thus
		$$
			c(x',y')\leq c(x,y)+\epsilon +\Lambda_Z|y-y'|+\omega_Z(|x-x'|)
		$$ 
		where $\omega_Z(r)\rightarrow 0$ as $r\rightarrow 0$. Since $x'$, $y'$, $\epsilon$ and $Z$ are arbitrary this shows that $c$ is continuous.
		Moreover, notice that \eqref{eqn:Lip-c} also shows, by letting $\epsilon \to 0$, that $c$ is Lipschitz-continuous with respect to $y$.
	 \end{proof}
	
	Now, let us apply Kantorovich duality \cite{V1}, which consists of considering the maximization problem: 
	\begin{align}\label{eqn:MKDualCost}
		D_1(\mu, \nu):= \sup_{ \psi, \phi \in C(\R^n);  \psi\oplus (-\phi) \le c } % \psi(y)-J_0(x)\leq c(x,y) \forall x,y \in \R^n}
		\left\{ \int_{\R^n}\psi(y)\nu(dy)-\int_{\R^n}\phi(x)\mu(dx)\right\}.
	\end{align}
	The following result is standard. See for example, Theorem 2.9 in \cite{V1}. The Lipschitz regularity of $\psi$ is verified by the representation of $\psi$ as the infimum of Lipschitz functions, $\psi(y)=\inf_x\{c(x,y)+\phi(x)\}$.
	\begin{proposition}\label{prop:MKattainment}
		Assume the hypotheses {\normalfont \ref{itm:continuity}}, {\normalfont \ref{itm:controllable}}, and {\normalfont \ref{itm:compact}}. Then there is $\pi^*\in \Pi(\mu,\nu)$ that attains the optimal value $V(\mu,\nu)$ and there exist potentials $\psi^*$, $\phi^*$ that attain the value $D_1(\mu, \nu)$. Moreover, $\psi^*$ is Lipschitz continuous, and
		\begin{align} \nonumber
			&V(\mu, \nu) = D_1(\mu, \nu),\\ \label{eqn:pi-a.e.-equal-c}
 			&\psi^*(y)-\phi^*(x)=c(x,y)  \quad \hbox{for $\pi^*$-a.e. $(x, y)$.}
		\end{align}
	\end{proposition}

	Now, we relate the dual potentials with the dynamic programming principle:
	Given $\psi$ upper semi-continuous, we define the `value' function  $J_\psi$ by 
	\begin{align} \label{eqn:dynamicProgramming}
		J_\psi(t,q)=\sup_{\tau,A(\cdot)}\Big\{\psi\big(\gamma(\tau)\big)-\int_t^{\tau} K(s,\gamma(s),A(s)\big)ds;\  \gamma(s)= \int_t^sf\big(\gamma(r),A(r)\big)dr+q\Big\}.
	\end{align} 
	From this definition, we see that for any $t_1\le t_2$, and $\gamma(s) = \int_{t_1}^s f(\gamma(r), A(r) ) dr + \gamma(t_1)$, $A(\cdot) \in \mathbb{A}$, 
	\begin{align}\label{eqn:JpsiMono}
 		J_\psi (t_2, \gamma(t_2)) - J_\psi(t_1, \gamma(t_1))   \le  \int_{t_1}^{t_2} K\Big(s, \gamma(s) , A(s) \Big) ds.
	\end{align}
	From the dynamic programming principle, $J_\psi$ is the unique lower semi-continuous viscosity solution to the Hamilton-Jacobi-Bellman type quasi-variational inequality,
	\begin{align} \label{eqn:HJB}
		\min\left\{\begin{array}{r} -\psi(q)+J_\psi(t,q), \\ -\frac{\partial}{\partial t} J_\psi(t,q)-{H}\big(t,q,\nabla J_\psi(t,q)\big)\end{array}\right\}=0,
	\end{align}
	where the Hamiltonian is $H(t,q,p)=\sup_{A\in \mathbb{A}}\Big\{p\cdot f(q,A)-K(t,q,A)\Big\}$, and the negative sign is introduced to be consistent with the definition of viscosity solutions.
	The equation \eqref{eqn:dynamicProgramming} plays the role of the `double convexification' procedure of optimal transportation, along with the replacement of $\psi$ by the maximal function satisfying $\psi\leq J_\psi$, given by $\psi(q)=\inf_t J_\psi(t,q)$.
	Recall  that we defined the dual problem in the introduction as
	\begin{align}\label{eqn:Dmunu}
 		D(\mu,\nu) := \sup_{\psi}\Big\{\int_{\R^n}\psi(y)\nu(dy)-\int_{\R^n}J_\psi(0,x)\mu(dx);\ J_\psi\ {\rm solves}\ (\ref{eqn:HJB})\Big\}.
	\end{align}
	\begin{proposition}\label{prop:HJB}
		Under the same assumptions of Proposition{\normalfont~\ref{prop:MKattainment}}, we have that 
		$$D(\mu,\nu) =D_1(\mu,\nu)= V(\mu, \nu).$$ 
		Moreover, there is an optimal $\psi^*$ such that the pair $(\psi^*(\cdot)=\inf_tJ_{\psi^*}(t,\cdot),  \phi^*(\cdot)=J_{\psi^*}(0, \cdot))$ is  an optimizer for $D_1(\mu, \nu)$, and hence satisfies \eqref{eqn:pi-a.e.-equal-c}. 
	\end{proposition}
	\begin{proof} 
		Let $(\psi, \phi)$ be an optimized Kantorovich pair for $D_1(\mu, \nu)$. In other words,  we have that 
		\begin{align}\label{eqn:MKConstraint}
			\psi(y)-\phi (x)\leq c(x,y) \quad \forall\ x, y \in \R^n, 
		\end{align}
		\[
			\psi(y)=\inf_{x\in \R^n} c(x, y)+\phi(x) \quad {\rm and} \quad\phi(x)=\sup_{y\in \R^n} \psi(y)-c(x, y).
		\]	
		Now, from  \eqref{eqn:dynamicProgramming}, the pair $(\psi, J_\psi (0, \cdot))$ obviously satisfies \eqref{eqn:MKConstraint}, thus is a Kantorovich pair for $D_1(\mu, \nu)$.  Since    $(\psi, \phi)$  is an optimized Kantorovich pair, then $J_\psi (0, \cdot)\le \phi (\cdot).$
		Therefore, the integral in \eqref{eqn:MKDualCost} for $D_1 (\mu, \nu)$ gets increased when replacing $\phi$ with $J_\psi$. Similarly, this integral increases when replacing $\psi$ with $\psi^*(\cdot)=\inf_tJ_\psi(t,\cdot)$. 
		This shows that 
		$
 			D_1 (\mu, \nu) = D(\mu,\nu) %:= \sup_{\psi}\Big\{\int_{\R^n}\psi(y)\nu(dy)-\int_{\R^n}J_\psi(0,x)\mu(dx);\ J_\psi\ {\rm solves}\ (\ref{eqn:HJB})\Big\}.
		$
		and that $(\psi^*, J_{\psi^*})$ is an optimized Kanotorovich pair.  
	\end{proof}

\section{An Eulerian Formulation}\label{sec:Eulerian}	

	We now introduce the Eulerian formulation which has cost $W(\mu,\nu)$, and its dual formulation with value $E(\mu,\nu)$.  We then prove that 
	$$ 
		D(\mu,\nu) = E(\mu, \nu) \leq W(\mu,\nu) \leq V(\mu,\nu).
	$$
	Taking into accounts the duality proven in Section~\ref{sec:Monge-Kantorovich}, all these values are then equal.

	Two key notions are the \emph{density process} $\rho:\R^+\rightarrow \mathcal{M^+}(\R^{n}\times \mathbb{A})$ and the \emph{stopping distribution} $\tilde{\rho}\in \mathcal{M^+}(\R^+\times \R^n)$, where $\mathcal{M}^+$ is the cone of non-negative finite Radon measures on the given metric space.  We assume $\rho$ is (weak*) measurable in the time variable. 
	The constraint on trajectories becomes the continuity equation \eqref{eqn:BB-continuity}
 	for $\rho$ and $\tilde{\rho}$, which we express in the weak form:
	\begin{align}
		\label{eqn:continuity} 
		&\int_{\R^n}\int_{\R^+} w(t,q)\tilde{\rho}(dt,dq)\\
		&\quad = \int_{\R^n} w(0,x)\mu(dx)\nn
 		+ \int_{\R^+} \int_{\mathbb{A}}\int_{\R^{n}}\Big[\frac{\partial}{\partial t}w(t,q)+f(q,A)\cdot \nabla w(t,q)\Big]\rho(t,dq,dA)dt
	\end{align}
	for all  $w \in C^1(\R \times \R^n)$.  The target constraint is given by 
	\begin{align}
		\label{eqn:Eulerian_target} 
		\int_{\R^n} u(y)\nu(dy) =&  \int_{\R^n}\int_{\R^+} u(q)\tilde{\rho}(dt,dq)\quad \hbox{for all $u\in C(\R^n).$}
	\end{align}
	If a pair $(\rho,\tilde{\rho})$ satisfies (\ref{eqn:continuity}) and (\ref{eqn:Eulerian_target}) then we say it is \emph{admissible}.

	We now consider the cost to be
	\begin{align} \label{eqn:Eulerian_cost} 
		W(\mu,\nu) := \inf\Big\{ \int_{\R^+}\int_{\mathbb{A}}   \int_{\R^{n}} K(t,q,A)\rho(t,dq,dA)dt ; \, \hbox{ $(\rho,\tilde{\rho})$ is admissible} \Big\}.
	\end{align}

	The following proposition associates to any transport plan, an approximating admissible density process. 
	\begin{proposition}\label{prop:Eulerian_embedding} 
		Suppose {\normalfont \ref{itm:continuity}}, {\normalfont \ref{itm:controllable}} and {\normalfont \ref{itm:compact}} hold.  If $\pi$ is a transport plan in $\Pi(\mu, \nu)$, then 
		for any $\epsilon>0$ there is an admissible density process and stopping distribution $(\rho,\tilde{\rho})$ such that
		\begin{align*}
  			\int_{\R^+}\int_{\mathbb{A}}   \int_{\R^{n}} K(t,q,A)\rho(t,dq,dA)dt \le \int_{\R^n}\int_{\R^n} c(x,y) \pi(dx,dy)+\epsilon
		\end{align*}
		As a consequence, 
		$
			W(\mu, \nu)\leq V(\mu,\nu).
		$
	\end{proposition}

	\begin{proof}
		We must first approximate the transport plan $\pi$ by nearly optimal trajectories with $(t,\gamma(\cdot),A(\cdot))$ remaining in a compact set.  
		Fix $\epsilon >0$. By the definition of $c(x, y)$ \eqref{eqn:MKCost}, for each 
		$x$ and $y$, we can find $A(\cdot;x,y)$ with values in a compact subset of $\mathbb{A}$ and $\tau(x,y)$ such that the trajectory  
		$$
			\gamma(t; x,y): =\int_{0}^{t} f(\gamma(s;x,y), A(s; x, y) ) ds + x
		$$
		satisfies $\gamma(\tau(x,y),x,y)=y$ and the cost satisfies  $$\int_0^{\tau(x,y)} K(t,\gamma(t;x,y),A(t;x,y)\big)dt \le c(x,y)+\frac{\epsilon}{2}.$$
		By the continuity of $c$ from Lemma \ref{lem:lsc}, there is a neighborhood $O$ of $(x,y)$ such that for $(x', y')\in O$ the trajectory beginning at $x'$ with the same control policy $A(\cdot;x,y)$ is within $\epsilon$ of optimality, after concatenating with a trajectory to fix the endpoint at $y'$.  With this construction on a finite open cover of the compact supports of $\mu$ and $\nu$, we have that all trajectories and controls remain in a compact set and the times $\tau(x,y)$ are uniformly bounded.
		
		All of $\gamma(\cdot;x,y)$, $\tau(x,y)$ and $A(\cdot;x,y)$ can be chosen to be measurable in $x$ with respect to $\mu$ and measurable in $y$ with respect to $\nu$ and are thus measureable with respect to $\pi$. Moreover, from \ref{itm:controllable}, the trajectories $\gamma(\cdot;x,y)$ are  Lipschitz continuous in time on a bounded time interval. 
%		From \ref{itm:controllable}, the trajectories $\gamma(\cdot;x,y)$ are  Lipschitz continuous in time on a bounded time interval.  All of $\gamma(\cdot;x,y)$, $\tau(x,y)$ and $A(\cdot;x,y)$ are measurable in $x$ with respect to $\mu$ and measurable in $y$ with respect to $\nu$ and are thus measureable with respect to $\pi$. %Furthermore, we may assume that all trajectories remain in a compact subset of $\R^n\times \R^+$ by \ref{itm:compact}. 
		We now define $\tilde{\rho}$ (by the Riesz representation theorem) such that
		\begin{align}
			\int_{\R^n}\int_{\R^+} \beta(t,q)\tilde{\rho}(dt,dq)=\int_{\R^{n}}\int_{\R^n} \beta\big(\tau(x,y),y\big)\pi(dx,dy)\nn
		\end{align}
		for all continuous and compactly supported $\beta$. We also define $\rho$ at time $t$ in such a way that
		\begin{align}
			\int_{\mathbb{A}}\int_{\R^n} \eta(q,A)\rho(t,dq,dA)=\int_{\R^{n}}\int_{\R^n} 1\big\{t<\tau(x,y)\big\}\eta\big(\gamma(t;x,y),A(t;x,y)\big)\pi(dx,dy)\nn
		\end{align}
		for all continuous and compactly supported $\eta$. Since $\pi$ is nonnegative, we see that $\rho, \tilde \rho$ are also nonnegative measures, and by construction, $\tilde{\rho}$ satisfies (\ref{eqn:Eulerian_target}) and $\rho$ is measurable in time.  We consider a smooth test function $w:\R\times \R^n \to \R$, and (dropping the notational dependence of $\tau,A(\cdot),\gamma(\cdot)$ on $x,y$), we  see that 
	\begin{align}
		&\ \int_{\R^+}\int_{\mathbb{A}} \int_{\R^{n}}\Big[-\frac{\partial}{\partial t}w(t,q)-f(q,A)\cdot \nabla w(t,q)\Big]\rho(t,dq,dA)dt\nn\\
		&\qquad \qquad \qquad =\ \int_{\R^+}\int_{\R^{n}}\int_{\R^n}1\big\{t<\tau\big\}\left[ -\frac{d}{dt}w\big(t,\gamma(t)\big)\right]\pi(dx,dy)dt\nn \\
		&\qquad \qquad \qquad =\ -\int_{\R^{n}}\int_{\R^n} w\Big(\tau,\gamma(t)\Big)\pi(dx,dy)+\int_{\R^{n}}\int_{\R^n}w(0,x)\pi(dx,dy)\nn\\
		&\qquad \qquad \qquad =\ -\int_{\R^n}\int_{\R^+} w(t,y)\tilde{\rho}(dt,dy)+\int_{\R^n}w(0,x)\mu(dx)\nn,
	\end{align}
	which is our version of the continuity equation (\ref{eqn:continuity}). Note that the first two equations above are possible as $\gamma$ is Lipschitz in $t$ so $t\mapsto w(t,\gamma(t))$ is $W^{1,\infty}$, and $t\mapsto \rho(t, \cdot, \cdot)$ is measurable. Therefore, the pair $(\rho, \tilde \rho)$ is an admissible density process. Now, a similar calculation shows that for the corresponding costs, we have:
	\begin{align}
		\int_{\R^+} \int_{\mathbb{A}}  \int_{\R^{n}} K(t,q,A)\rho(t,dq,dA)dt\nn
		=&\ \int_{\R^+}\int_{\R^{n}}\int_{\R^n} 1\{t<\tau\}K\big(t,\gamma(t),A(t)\big)\pi(dx,dy)dt\nn\\
		=&\ \int_{\R^{n}}\int_{\R^n}\left[\int_0^{\tau} K\big(t,\gamma(t),A(t)\big)dt\right] \pi(dx,dy)\nn\\
		\leq& \int_{\R^{n}}\int_{\R^n}c(x,y)\pi(dx,dy)+\epsilon.\nn
	\end{align}
	As $\epsilon$ is arbitrary,  we get that $W(\mu,\nu)\leq V(\mu,\nu)$.
	\end{proof}

	From the duality principle of linear programming, we can formulate the dual problem to the Eulerian formulation as follows:
	\begin{align} \label{eqn:dualCost}
		\hbox{ Maximize} \quad 	\int_{\R^n}\psi(y)\nu(dy)-\int_{\R^n}J(0,x)\mu(dx),
	\end{align}
	over $\psi \in C(\R^n)$ and $J\in C^1(\R \times \R^n)$ subject to
	\begin{align} \label{eqn:dual_constraints}
		\psi(q) - J(t,q)\leq&\ 0,\ &\ {\rm for\ all}\ (t,q)\in \R^+\times \R^n\\
		\frac{\partial}{\partial t} J(t,q)+f(q,A)\cdot \nabla J(t,q)\leq&\ K(t,q,A),\ &\ {\rm for\ all}\ (t,q,A)\in  \R^+\times \R^{n}\times  \mathbb{A}. \nn
	\end{align}
 	Notice that \eqref{eqn:dual_constraints} means that $J$ is a supersolution to \eqref{eqn:HJB}. 

	We let $E(\mu,\nu)$ denote the optimal value of this dual problem. 
	\begin{proposition}\label{prop:Eulerian_duality}
		For all $(J,\psi)\in  C^1(\R^+ \times \R^n)\times C(\R^n)$ that satisfy {\normalfont (\ref{eqn:dual_constraints})}, and any $(\rho,\tilde{\rho})$ that satisfy {\normalfont (\ref{eqn:continuity})} and {\normalfont (\ref{eqn:Eulerian_target})}, we have 
		\begin{align} \label{eqn:weakDuality}	
			\int_{\R^n}\psi(y)\nu(dy)-\int_{\R^n}J(0,x)\mu(dx)\leq\int_{\R^+}\int_{\mathbb{A}}\int_{\R^{n}} K(t,q,A)\rho(t,dq,dA)dt.
		\end{align}  
		 In particular,  $E(\mu, \nu) \le W(\mu, \nu)$.  
	\end{proposition}

	\begin{proof}
		This weak duality inequality is immediate from plugging in $J$ and $\psi$ as test functions in (\ref{eqn:continuity}) and (\ref{eqn:Eulerian_target}).  
	\end{proof}

	This dual problem \eqref{eqn:dualCost}-\eqref{eqn:dual_constraints}  is indeed equivalent to the duality studied in Section \ref{sec:Monge-Kantorovich}, as we can reduce \eqref{eqn:dual_constraints} to \eqref{eqn:HJB}. To see this,  consider the pointwise  minimum of $J(t,q)$ amongst all supersolutions and the pointwise  maximum of $\psi$ among all functions  less than or equal to $J$. These yield   a solution to (\ref{eqn:HJB}). 

	\begin{proposition}\label{prop:dual_equivalence}
	 	Given $\psi\in C(\R^n)$ uniformly bounded, then the infimum  %maximum
		 over all $J$ that satisfy {\normalfont (\ref{eqn:dual_constraints})} is attained by the function $J_\psi$ defined in {\normalfont (\ref{eqn:dynamicProgramming})}. Equivalently, this is the unique viscosity solution to {\normalfont (\ref{eqn:HJB})}, and therefore, the dual value satisfies $E(\mu,\nu)=D(\mu,\nu)$.
	\end{proposition}

	%Add more detail or at least reference to this proof. AP
	\begin{proof}
	 Recall the definition of $J_\psi$ from \eqref{eqn:dynamicProgramming}, which is the unique viscosity solution to \eqref{eqn:HJB}. 	Therefore,  by the comparison principle, if ${J}$ satisfies \eqref{eqn:dual_constraints}, then  ${J}(t,q)\ge J_\psi(t,q)$ and we have $E(\mu,\nu)\leq D(\mu,\nu)$. Perron's method then yields that $J_\psi$ is the infimum of all supersolutions $J$; see \cite{bardi2008optimal}. Comparing with the definition of $D(\mu, \nu)$ in \eqref{eqn:Dmunu}, this proves that the supremum of the value over $J$ is attained at $J_\psi$ and that $E(\mu, \nu) = D(\mu,\nu)$. 
	\end{proof}
	Now we can verify that the problems are indeed equivalent.
	\begin{theorem}\label{thm:Eulerian_equality}
		Given the hypotheses {\normalfont \ref{itm:continuity}}, {\normalfont \ref{itm:controllable}}, {\normalfont \ref{itm:compact}}, then  the following hold: %\eqref{eqn:BB} holds, namely, 
		\begin{equation} 
			D(\mu,\nu) = E(\mu, \nu) = W(\mu,\nu) = V(\mu,\nu).
		\end{equation}
	\end{theorem}
	\begin{proof}
		Proposition \ref{prop:Eulerian_embedding} implies that $W(\mu,\nu)\leq V(\mu,\nu)$ and Proposition \ref{prop:Eulerian_duality} implies $E(\mu,\nu)\leq W(\mu,\nu)$.  Proposition \ref{prop:dual_equivalence} yields that $E(\mu,\nu)=D(\mu,\nu)$. The rest follows from the fact that   $V(\mu,\nu)=D(\mu,\nu)$ shown in Section~\ref{sec:Monge-Kantorovich}.
	\end{proof}

\section{Necessary Optimality Conditions}\label{sec:optimality_conditions}

	The next stage of our analysis is to relate the Hamilton-Jacobi-Bellman quasi-variational inequality (\ref{eqn:HJB}) back into optimality criteria for paths involved in the primal problem.

	\begin{theorem}\label{thm:Pontryagin}
		Suppose {\normalfont \ref{itm:continuity}}, {\normalfont \ref{itm:controllable}} and {\normalfont \ref{itm:compact}} hold and also that $(q,A)\mapsto \nabla f$ is continuous as is $(t,q,A)\mapsto \nabla K$. We let $\pi$, $\psi$ and $J_\psi$ be optimal for \eqref{eqn:transport_problem} and \eqref{eqn:Dmunu}, resp., and suppose that the optimal control problem  \eqref{eqn:MKCost} for the cost $c$ is attained for $\pi$-a.e. $(x,y)$ at $\tau^{x,y}$ and $A^{x,y}(\cdot)$ that are measurable with respect to $x$ and $y$, and for a.e.\ $(x,y)$, $A^{x,y}(\cdot)$ is piecewise continuous in time. We let $\gamma^{x,y}(\cdot)$ denote the corresponding trajectories.  Then the following hold: 
		\begin{enumerate}
			\item  If $\gamma(\cdot),\tau$ and $A(\cdot)$ attain the optimal cost $c(x,y)$ with $\gamma(0)=x$ and $\gamma(\tau)=y$, then $J_\psi(\tau,y)=\psi(y)$.
			
\item 
There is $p\in C([0,\tau]; \R^n)$ that solves, at times when $A(\cdot)$ is continuous,
				\begin{align}\label{eqn:costate}
					\dot{p}(t) = -p(t)^\top\nabla f\big(t,\gamma(t),A(t)\big)+\nabla K\big(t,\gamma(t),A(t)\big),
				\end{align}
				and such that for a.e.\ $t\in[0,\tau]$
				\begin{align}\label{eqn:maximum_principle}
					H\big(t,\gamma(t),p(t)\big)=p(t)\cdot f\big(\gamma(t),A(t)\big)-K(t,\gamma(t),A(t)\big),
				\end{align}
				while the following transversality condition holds:
				\begin{align}\label{eqn:transversality}
					H\big(\tau,y,p(\tau)\big)=0.
				\end{align}

			\item	If $\psi$ is differentiable at $y$, then $p(\tau)=\nabla \psi(y)$, and if $J_\psi$ is differentiable at $(t,\gamma(t))$, then $p(t)=\nabla J_\psi(t,\gamma(t))$.
		\end{enumerate}
	\end{theorem}

	\begin{proof}
		Suppose $\gamma^{x,y}$, $A^{x,y}$, and $\tau^{x,y}$ attain the infimum in the expression of $c(x,y)$. Then, for $\pi$-almost every $(x,y)$ and $t\leq \tau^{x,y}$,
		$$
			J_\psi\big(t,\gamma^{x,y}(t)\big) = \psi(y)-\int_t^{\tau^{x,y}} K\big(s,\gamma^{x,y}(s),A^{x,y}(s)\big)ds,
		$$
		from which follows that $J_\psi(\tau^{x,y},y)=\psi(y)$.

		The Pontryagin maximum principle \cite{P} implies that if $\gamma$, $\tau$, $A$ minimize the free end time optimal control problem then there is $p$ that satisfies (\ref{eqn:costate}), (\ref{eqn:maximum_principle}) and (\ref{eqn:transversality}).  For almost every $(x,y)$  in the support of $\pi$, then $\gamma$, $\tau$ and $A$ also minimize
		$$
			J_\psi\big(t,\gamma(t)\big)-\psi\big(\gamma(\tau)\big)+\int_t^\tau K\big(s,\gamma(s),A(s)\big)ds 
		$$
		amongst trajectories with free position at times $t$ and $\tau$.  It then also follows from the Pontryagin maximum principle and differentiability of $\psi(y)$ and $J_\psi(t,\gamma(t))$ that $p(\tau)=\nabla\psi(y)$ and $p(t)=\nabla J(t,\gamma(t))$.
	\end{proof}

	\begin{remark}
		Since the optimal control is generally not attained at  a measurable map, we could compactify the problem by using the notion of Young measures, which are weak* measurable maps $t\mapsto \alpha_t\in \mathcal{P}(\mathbb{A})$ (probability measures). As an example consider the case where $\mathbb{A}$ is a finite set, and the optimal direction does not align with $f(q,A)$ for any of these points.  A nearly optimal control will rapidly oscillating between points in $\mathbb{A}$ and the optimal control is obtained by a Young measure in $\mathcal{P}(\mathbb{A})$, which may be interpreted as a randomized choice of control at each instant. 
		In this case, we 
		consider controlled trajectories of the form
		$$
			\gamma(t) = \int_0^t \int_\mathbb{A} g\big(\gamma(s),A\big)\alpha_s(dA)ds + x.
		$$
		This formulation provides an intermediary step between the primal problem and its Eulerian formulation.  However, we choose instead to work with the weak Eulerian formulation of Section {\normalfont\ref{sec:Eulerian}}, similar to what is done in {\normalfont\cite{evans2001effective}}, and we later work with the stronger assumptions that the cost is given by a convex Lagrangian.
	\end{remark}
	We now introduce hypotheses that will imply attainment of the Eulerian formulation.

	\begin{enumerate}[label=\textbf{H\arabic*}] \setcounter{enumi}{2}
		\item \label{itm:ACoercive}
			%We assume that 
			The set of controls $\mathbb{A}\subset \R^m$ is closed, $f(q,A)$ is uniformly bounded in $q$ for each $A$, $|f(q,A)|\leq \Gamma_0(1+|A|)$, and $K(t,q,A)$ satisfies
			$$
				\liminf_{A\rightarrow \infty} \frac{K(t,q,A)}{|A|} = \infty.
			$$

		\item \label{itm:tCoercive} The following coercivity in time holds:
			%We assume that 
			$$
				\liminf_{T\rightarrow \infty}\int_0^T \inf_{q,A} K(t,q,A)dt =\infty.
			$$
	\end{enumerate} 
	We will abbreviate {\normalfont\ref{itm:continuity}} - {\normalfont\ref{itm:tCoercive}} to include hypotheses {\normalfont\ref{itm:continuity}}, {\normalfont\ref{itm:controllable}}, {\normalfont\ref{itm:compact}}, {\normalfont\ref{itm:ACoercive}} and {\normalfont\ref{itm:tCoercive}}.

	\begin{theorem}\label{thm:Eulerian_verification}
		Suppose {\normalfont\ref{itm:continuity}} - {\normalfont\ref{itm:tCoercive}} hold, then the Eulerian formulation is minimized by some pair $(\rho,\tilde{\rho})$.
		For optimal $(\psi,J_\psi)$ (cf.\ Proposition {\normalfont\ref{prop:MKattainment}} and Proposition {\normalfont\ref{prop:dual_equivalence}}):
		\begin{enumerate}
			\item \label{itm:time_slackness}
			$J_\psi(t,q)=\psi(q)$ for $\tilde{\rho}$ almost every $(t,q)$;
			\item If $J_\psi\in C^1(\R^+\times \R^n)$, then for almost every $t$, and $\rho(t,\cdot,\cdot)$ almost every $(q,A)$, $f(q,A)\in D_pH(t,q,\nabla J_\psi(t,q))$, where $D_pH$ denotes the subdifferential of $H$ with respect to $p$.
		\end{enumerate}  
	\end{theorem}

	\begin{proof}
		The attainment of the minimum at $(\rho,\tilde{\rho})$ is by a compactness argument using coercivity and tightness.  We fix a cost $M\in \R$ and show that the set of $(\rho,\tilde{\rho})$ with cost less than $M$ is tight. %We now 
		First, let $\epsilon>0$ be arbitrary. Given assumption \ref{itm:tCoercive} there is $T$ such that
		$$
			\int_0^{T} \inf_{q,A} K(t,q,A)dt > M/\epsilon.
		$$
		If $\tilde{\rho}([T,\infty)\times \R^n) \geq \epsilon$ then $\rho(T,\R^n,\mathbb{A})\geq \epsilon$, because
		$$
			\int_0^T\int_{\R^n}\tilde{\rho}(dq,dt)+\int_{\mathbb{A}}\int_{\R^n}\rho(T,dq,dA)=1,
		$$
		which implies that the cost of $(\rho,\tilde{\rho})$ is greater than $M$.
		Similarly, by \ref{itm:ACoercive} there is $R_1$ such that $K(t,q,A)>|A|M(1+\Gamma_0)/\epsilon$ if $|A|>R_1$, which implies
		$$
			\int_{\R^+} \int_{|A|>R_1}\int_{\R^n}|A|(1+\Gamma_0)\rho(t,dq,dA)dt< \epsilon.
		$$
		We let $\Gamma_1=\sup_{q\in \R^n,|A|<R_1}|f(q,A)|$.  The amount of mass at time $t$ outside a ball of radius $R_2(t)=2(t\Gamma_1+D+1)$ is less than $\epsilon$, where the support of $\mu$ is contained in a ball of radius $D$. This can be shown from \eqref{eqn:continuity} by considering a test function, $w(t,q)$, that is zero on a ball of radius $R_2(t)/2$, one outside of a ball of radius $R_2(t)$, nonegative, and satisfies 
		$$
			\frac{\partial}{\partial t}w(t,q)\leq -|\nabla w(t,q)| \Gamma_1,
		$$
		and $|\nabla w(t,q)|\leq 1$ for $t\leq T$ (consider a smooth approximation to the piecewise interpolation from the function that is $0$ for $|q|\leq R_2(t)/2$ and $1$ for $|q|\geq R_2(t)$).
		Then \eqref{eqn:continuity} implies
		\begin{align}
			&\ \int_{|q|>R_2(T)}\int_{0}^T\tilde{\rho}(dt,dq)+\int_{\mathbb{A}}\int_{|q|>R_2(T)}\rho(T,dq,dA)\nn\\
			\leq&\ \int_0^T\int_{\mathbb{A}}\int_{\R^n}\Big[\frac{\partial}{\partial t}w(t,q)-f(q,a)\cdot \nabla w(t,q)\Big]\rho(t,dq,dA)dt\nn\\
			\leq&\ \int_0^T\int_{|A|>R_1}\int_{\R^n}\Big[|\nabla w|(1+\Gamma_0)|A|\Big]\rho(t,dq,dA)dt< \epsilon.\nn
		\end{align}

		% To show that $\int_{\R^+}\int_{|q|>R_2(t)}\tilde{\rho}(dq,dt)<\epsilon$, we now let $w$ be zero for $|q|\leq R_2(t)$ and otherwise $w(q)=|q-R_2|$. Then \eqref{eqn:continuity} implies that
		% $$
		% 	\int_{\R^n}\int_{0}^tw(q)\tilde{\rho}(ds,dq)\leq \frac{\epsilon}{2}+\int_{|q|>R_2}\int_{|A|<R_1}\Gamma(R_1)\rho(t,dq,dA)\leq 0.
		% $$
		These estimates show that the mass outside of $[0,T]\times B_{R_1}\times B_{R_2(T)}$ is less than $\epsilon$. This tightness ensures that a subsequence of 
		a minimizing
		sequence of $(\rho^i,\tilde{\rho}^i)$ converges in the weak* topology to some $(\rho,\tilde{\rho})$ that attains the minimum cost.

		We now let $\psi$ be optimal and consider a sequence $J^i\in C^1(\R^+\times \R^n)$ that satisfy \eqref{eqn:dual_constraints} and converge uniformly to $J_\psi$.  Then using $J^i$ and $\psi$ as test functions for \eqref{eqn:continuity} and \eqref{eqn:Eulerian_target} with optimal $(\rho,\tilde{\rho})$, we have that
		\begin{align}
			&\ \int_{\R^n}\psi(y)\nu(dy)-\int_{\R^n}J^i(0,x)\mu(dx)\nn\\
			\leq &\ \int_{\R^n}\int_{\R^+}\big[\psi(q)-J^i(t,q)\big]\tilde{\rho}(dt,dq) + \int_{\R^+}\int_\mathbb{A} \int_{\R^n} K(t,q,A)\rho(t,dq,dA)dt.\nn
		\end{align}
		In the limit as $i\rightarrow 0$ this shows that $0\leq \int_{\R^n}\int_{\R^+}[\psi(q)-J_\psi(t,q)]\tilde{\rho}(dt,dq)$, which along with $\psi(q)\leq J_\psi(t,q)$ proves statement \textit{(\ref{itm:time_slackness})}.

		If $J_\psi$ is in fact $C^1$, then using $J_\psi$ as a test function for \eqref{eqn:continuity} as above also implies that
		$$
			\frac{\partial}{\partial t}J_\psi(t,q)+f(q,A)\cdot \nabla J_\psi(t,q)-K(t,q,A)=0
		$$
		for almost every $t$ and $\rho$ almost every $q$ and $A$. It follows that at these points $A$ maximizes $f(q,A)\cdot \nabla J_\psi(t,q)-K(t,q,A)$, which occurs at $f(q,A)\in D_pH(t,q,\nabla J_\psi(t,q))$.  
	\end{proof}

	\begin{remark}\label{rem:Wasserstein}
		While $(\rho,\tilde{\rho})$ are convenient for the existence and weak optimality criteria, we can now move to a more familiar Eulerian perspective with the density defined informally by
		$$
			\eta(t,q) = \int_{\mathbb{A}}\rho(t,q,dA).
		$$
		The statement $f\in D_pH$ that we have proven in Theorem {\normalfont\ref{thm:Eulerian_verification}} can now be reinterpreted as $\eta$ satisfies a weak continuity inequality, formally equivalent to
		$$
			\frac{\partial}{\partial t} \eta(t,q) \leq -{\rm div} \Big(D_pH\big(t,q,\nabla J_\psi(t,q)\big)\eta(t,q)\Big),
		$$
		where we have replaced the stopping distribution $\tilde{\rho}$ with an inequality.
	\end{remark}

\section{When optimal end times are hitting times of a barrier}\label{sec:hitting_times}

	In this section, we study the structure of the solutions of the dual problem under different cases of time dependence as well as the implications for the optimal flows. In particular, our analysis will  focus on the viscosity solutions to (\ref{eqn:HJB}), and how their properties translate to the primal problem.  We consider three cases:
	\begin{description}
		\item[\hypertarget{itm:TS}{TS}] The time-stationary case: i.e., when $K$ is independent of time.
		\item[\hypertarget{itm:TC}{TC}] The time-compounded case: i.e., when $K$ is strictly increasing in time.
		\item[\hypertarget{itm:TD}{TD}] The time-discounted case: i.e., when $K$ is strictly decreasing in time.
	\end{description}
	
	\begin{proposition}\label{prop:dualStructure}
		We suppose {\normalfont \ref{itm:continuity}} and {\normalfont \ref{itm:controllable}}, and that the pair $(J_\psi,\psi)$ solve {\normalfont (\ref{eqn:HJB})} and $\psi(q)=\inf_{t\in \R^+}\{J_\psi(t,q)\}$. We consider the three cases of time dependence.  We show the following inequalities hold for certain values of $t$ in each case:
		\begin{align} \label{eqn:phiSubsolution}	
			-H\big(t,q,\nabla J_\psi(t,q)\big)\leq 0
		\end{align}
		and
		\begin{align} \label{eqn:psiSupersolution} 
			-H\big(t,q,\nabla \psi(q)\big)\geq 0.
		\end{align}

		\begin{enumerate}
			\item In the time-stationary case \hyperlink{itm:TS}{\rm \bf [TS]}, the dual solution $J_\psi$ is constant in time, $J_\psi(t,q)=\psi(q)$, and  {\normalfont (\ref{eqn:psiSupersolution})} holds for all $q$.
			\item In the time-compounded case \hyperlink{itm:TC}{\rm \bf [TC]}, the solution $J_\psi$ is non-increasing, the free boundary $s(q)=\inf\{t;J_\psi(t,q)=\psi(q)\}$ is lower semi-continuous, and %$J_\psi(t,q)=\sup\{J_\psi(s,q); s\geq t\}$. 
			Equation {\normalfont(\ref{eqn:psiSupersolution})} holds for $t\geq s(q)$ and {\normalfont(\ref{eqn:phiSubsolution})} holds for $t< s(q)$.
			%, where
			% $$s(q)=\inf\{t;J_\psi(t,q)=\psi(q)\}.$$ 
			\item In the time-discounted case  \hyperlink{itm:TD}{\rm \bf [TD]}, the solution $J_\psi$ is non-decreasing, the free boundary $s(q)=\sup\{t;J_\psi(t,q)=\psi(q)\}$ is upper semi-continuous, and %$\phi(t,q)=\inf\{J_\psi(s,q); s\geq t\}$. 
			Equation {\normalfont(\ref{eqn:psiSupersolution})} holds for $t\leq s(q)$ and {\normalfont(\ref{eqn:phiSubsolution})} holds for $t> s(q)$.
			%, where 
		%	$$s(q)=\sup\{t;J_\psi(t,q)=\psi(q)\}.$$ 
			
		\end{enumerate}

		Moreover, in both cases \hyperlink{itm:TC}{\rm \bf [TC]} and \hyperlink{itm:TD}{\rm \bf [TD]} it also holds that if $0<s(q)<\infty$ then
		\begin{align} \label{eqn:psiSolution}
			-H\big(s(q),q,\nabla \psi(q)\big)=0.
		\end{align}
		
	\end{proposition}

	\begin{proof}
	
		The stationary case \hyperlink{itm:TS}{\rm \bf [TS]} follows immediately from the uniqueness, given $\psi$, of viscosity solutions to (\ref{eqn:HJB}) and the symmetry as $t\mapsto t+\Delta t$.  A consequence of (\ref{eqn:HJB}) is that 
		$$
			-{H}\big(t,q,\nabla J_\psi(t,q)\big)\geq 0
		$$ 
		and since $J_\psi=\psi$, equation (\ref{eqn:psiSupersolution}) holds.

		For the time-compounded case \hyperlink{itm:TC}{\rm \bf [TC]}, we let 
		$\tilde{J}(t,q)=\sup\{J_\psi(s,q); s\geq t\}$, which clearly satisfies $\tilde{J}\geq J_\psi$ and $\tilde{J}$ is non-increasing.  To prove $\tilde{J}\leq J_\psi$, we show that $\tilde{J}$ is a subsolution to (\ref{eqn:HJB}), then use the comparison principle.
		The subsolution property follows from considering two cases: If $J_\psi(r,q)=\psi(q)$ for all $r\geq t$, then $\tilde{J}(t,q)=J_\psi(t,q)=\psi(q)$ and it is obviously a subsolution; Otherwise the supremum in the definition of $\tilde{J}$ is attained at a time $r^*\geq t$ where $J_\psi(r^*,q)>\psi(q)$. We then consider a test function $w$ that touches $\tilde{J}$ from above at $(t,q)$, and  let $\Delta t = r^*-t$ and $\hat{w}(r,q)=w(r+\Delta t,q)$.  Then $\hat{w}$ touches $J_\psi$ from above at $(r^*,q)$.  It follows from the definition of a viscosity subsolution that
		\begin{align}
			0\geq &\ -\frac{\partial}{\partial s}\hat{w}(r^*,q)-{H}\big(r^*,q,\nabla \hat{w}(r^*,q)\big)\nn\\
			\geq&\ -\frac{\partial}{\partial t}w(t,q)-{H}\big(t,q,\nabla w(t,q)\big)\nn,
		\end{align}
		where we have used Assumption \hyperlink{itm:TC}{\rm \bf [TC]} that $K$ is increasing and therefore the Hamiltonian is decreasing in time. This proves that $\tilde{J}$ is a viscosity subsolution to (\ref{eqn:HJB}).  We conclude that $\tilde{J}=J_\psi$, and that $J_\psi$ is non-increasing, which also proves (\ref{eqn:phiSubsolution}) when $J_\psi(t,q)>\psi(q)$.  Equation (\ref{eqn:psiSupersolution}) follows immediately from (\ref{eqn:HJB}) in the case that $J_\psi=\psi$ and hence $\psi$ touches $J_\psi$ from below thus $-H(t,q,\nabla\psi(q))\geq -H(t,q,\nabla J_\psi(t,q))\geq 0$, where the last inequality follows from that $J_\psi$ is constant in time for $t\geq s(q)$ but still satisfies \eqref{eqn:HJB}.  

		The case of \hyperlink{itm:TD}{\rm \bf [TD]} follows exactly the same argument replacing the supremums by infimums, subsolutions by supersolutions, and touching above by touching below. 

		That $s$ is lower (upper) semi-continuous follows from being the inf (sup) over a closed set.  
		For the final equation (\ref{eqn:psiSolution}), that $-H(s(q),q,\nabla \psi(q))\geq 0$ follows from (\ref{eqn:psiSupersolution}) with $t=s(q)$. For the other inequality, suppose to the contrary that there is a smooth test function $g$ with $g\geq \psi$, $g(\overline{q})=\psi(\overline{q})$, and $-H(s(\overline{q}),\overline{q},\nabla g(\overline{q}))>0$.  Then we can construct a function $G(t,q)=g(q)+f(t)$, which satisfies $G\geq \psi$,
		$$
			-\frac{\partial}{\partial t} G(t,q)-H\big(t,q,\nabla G(t,q)\big)\geq 0
		$$
		and $G(t,\overline{q})=\psi(\overline{q})$ for some $0\leq t<s(\overline{q})$ (or $t>s(\overline{q})$ in the case of \hyperlink{itm:TD}{\rm \bf [TD]}). In particular, for the case of \hyperlink{itm:TC}{\rm \bf [TC]}, we choose some $\delta>0$ and $f(t)$ that is zero for $t>s(q)-\delta$ and $f'(t)$ is sufficiently negative for $t<s(q)-\delta$ to ensure that $-f'(t)-H(t,q,g(q))>0$. Since $J_\psi$ is the infimum over all such supersolutions, we have that $J_\psi\leq G$, however, this contradicts the definition of $s$, which implies that $J_\psi(t,\overline{q})>\psi(\overline{q})$ for $t<s(\overline{q})$ (or $t>s(\overline{q})$ for \hyperlink{itm:TD}{\rm \bf [TD]}).
	\end{proof}

	This next lemma encodes the simple properties that will allow us to prove that the stopping plans are infact given by a single stopping time. First, we consider that stopping times are given by a function of the Hamiltonian trajectories $(\gamma,p)\mapsto \tau$.  In Theorem \ref{thm:Monge_map} this is used in the setting where the trajectories are uniquely determined by their initial condition, and the initial costate condition is determined by $\nabla J_\psi(0,\cdot)$.  The second setting we consider is where the trajectories are determined by their terminal conditions, in which case the terminal costate is given by $\nabla \psi(y)$ and the stopping time is the free boundary $s(y)$ defined in Proposition \ref{prop:dualStructure}.
	\begin{lemma}\label{lem:hitting_time}
		Given the same hypotheses as Theorem {\normalfont\ref{thm:Pontryagin}}, assuming either \hyperlink{itm:TC}{\rm \bf [TC]} or \hyperlink{itm:TD}{\rm \bf [TD]} :
		\begin{enumerate}
			\item Suppose that $\gamma(\cdot)$, $A(\cdot)$ are optimizers of $c(x,y)$ with $x\not=y$, and $p(\cdot)$ is a costate warranted by the Pontryagin maximum principle.  Then along the trajectory the Hamiltonian, $H(t,\gamma(t),p(t))$, strictly decreases in the case of \hyperlink{itm:TC}{\rm \bf [TC]} or strictly increases in the case of \hyperlink{itm:TD}{\rm \bf [TD]}.  In particular, for each Hamiltonian trajectory ($\gamma,p$) the end time $\tau$ is determined by the transversality condition \eqref{eqn:transversality}.
			\item Suppose that $\gamma$, $A$ are instead optimizers of $J(0,x)$, i.e.\ the free end time problem with terminal cost $-\psi(\cdot)$.  If $\gamma$ ends at $y$ where $\psi$ differentiable with end time $\tau>0$, then the end time is $\tau = s(y)$, i.e.\ the unique time such that $H(t,y,\nabla\psi(y))=0$.
		\end{enumerate}
	\end{lemma}
	\begin{proof}
		A consequence of the Pontryagin maximum principle is that for almost every $t<\tau$,
		\begin{align}
			&\ \frac{d}{dt} H\big(t, \gamma(t),p(t)\big) \nn\\
			=&\ -\frac{\partial}{\partial t} K\big(t,\gamma(t),p(t)\big)\nn \\
			&\ +\big[p^\top(t) \nabla f\big(\gamma(t),A(t)\big)-\nabla k\big(t,\gamma(t),A(t)\big)\big]\dot{\gamma}(t)+\dot{p}(t) f\big(\gamma(t),A(t)\big),\nn
		\end{align}
		and the last line is zero from the costate equation \eqref{eqn:costate}. Given \hyperlink{itm:TC}{\rm \bf [TC]} , $-\frac{\partial}{\partial t} K<0$ and $H$ decreases along the trajectory, and given \hyperlink{itm:TD}{\rm \bf [TD]}  the opposite holds. Clearly, this implies that $H(t,\gamma(t),p(t))=0$ can occur for at most one time.

		In the second case, $p(\tau)=\nabla \psi(\gamma(\tau))$ from the Pontryagin maximum principle as used in Proposition \ref{prop:dualStructure}, and if $\gamma(\tau)=y$, the equation $H(\tau,y,\nabla \psi(y))=0$ uniquely determines $\tau=s(y)$.
	\end{proof}

	\begin{remark}
		The Pontryagin maximum principle can also be stated for suitable sub/super derivatives of $\psi$ and $J_\psi$.  The lack of differentiablity of $\psi$ at $y$ corresponds to the existence of multiple optimal trajectories terminating at $y$ with possibly different end times satisfying $H(\tau,y,\beta)=0$ for $\beta\in \partial \psi(y)$. Similarly, lack of differentiability of $J_\psi$ at $(t,q)$ corresponds to multiple optimal trajectories emitting from $\gamma(t)=q$ with different $p(t)$ values.
	\end{remark}

	\begin{remark}
		The time monotonicity of $J_\psi$ of Proposition {\normalfont\ref{prop:dualStructure}} and characterization of stopping times of Lemma {\normalfont\ref{lem:hitting_time}} is closely related to the geometric pathwise monotonicity of {\normalfont\cite{beiglboeck2017optimal}}. Beiglb\"{o}ck et al.\ consider \emph{stop-go} pairs of trajectories $(\gamma_1(\cdot),\gamma_2(\cdot))$ where $\gamma_1(t_1)=\gamma_2(t_2)=y$ for $t_1<t_2$.  In the case of \hyperlink{itm:TC}{\rm \bf [TD]}, if $\gamma_1(\cdot)$ stops at $t_1$, then $\gamma_2(\cdot)$ must also stop at $t_2$. Another poof of this principle uses the dynamic programming principle.  Consider that $\gamma_2(\cdot)$ is optimal and continues until time $\tau_2>t_2$, then
		\begin{align}
			J_\psi(t_2,y) =&\ \psi(y)-\int_{t_2}^{\tau_2}L\big(t,\gamma_2(t)\big)dt\nn\\
			<&\ \psi(y)-\int_{t_1}^{\tau_2-t_2+t_1}L\big(t,\gamma_2(t+t_2-t_1)\big)dt\nn\\
			\leq&\ J_\psi(t_1,y).\nn
		\end{align}
		This shows that if $\gamma_2(\cdot)$ continues at time $t_2$, then $\gamma_1(\cdot)$ must also continue at time $t_1$ because $J_\psi(t_1,y)>\psi(y)$, or equivalently if $\gamma_1(\cdot)$ stops at time $t_1$, then $\gamma_2(\cdot)$ must also stop at time $t_2$.
	\end{remark}

	We will now make some assumptions directly on the Hamiltonian that are sufficient for our analysis. %Add  references for further enquiry? AP
	\begin{enumerate}[label=\textbf{H\arabic*}] \setcounter{enumi}{4}
		\item \label{itm:HCoercive} We also assume that the Hamiltonian is smooth and satisfies $H(t,q,p)\sim |p|^\beta$ uniformly in $t$ and $q$ for $1<\beta<\infty$.

		\item \label{itm:Aconvex} We assume that $p\cdot f(q,A)-K(t,q,A)$ admits a unique maximizer $A^*$ for all $t$, $q$, and $p$, and $(t,q,p)\mapsto A^*$ is continuous.
	\end{enumerate}

	We recall the notion of a Hamiltonian flow $\chi_H:\R^+\times\R^+\times \R^n\times\R^n\rightarrow \R^n\times \R^n$ as $\chi_H(t_1,t_2,q,\beta) = (\gamma(t_2),p(t_2))$ where $\gamma$ and $p$ solve
	\begin{align}
		\dot{\gamma}(t) =&\ D_pH\big(t,\gamma(t),p(t)\big)\nn\\
		\dot{p}(t)=&\ -D_qH\big(t,\gamma(t),p(t)\big),\nn
	\end{align}
	with $\gamma(t_1)=q$ and $p(t_1)=\beta$.  In general this Hamiltonian flow may be set-valued, but under hypothesis \ref{itm:HCoercive} the solutions are unique.
	We let 
	$$
		\chi_{_{H,J_\psi}} (t, x)= \pi^*\chi_H\big(0,t, x, \nabla J_\psi (0, x)\big)
	$$ 
	be the spatial flow with the initial momentum determined by $\nabla J_\psi(0,x)$. Here $\pi^*:\R^n\times \R^n \rightarrow \R^n$ is the canonical projection.  We let $\tau(x)$ denote the unique stopping time determined by the trajectory $(\gamma,p)$ cf.\ Lemma \ref{lem:hitting_time}.

	We also consider the backward flow with terminal end time given by $s(y)$ where for $t<s(y)$ we let
	$$
	\chi_{_{H,\psi,s}} (t, y) = \pi^*\chi_H\big(s(y),t, y, \nabla \psi (y)\big)
	$$
	be the backwards flow with terminal condition $\gamma(s(y))=y$ and $p(s(y))=\nabla \psi(y)$. 
	 
	We note the relationship of the end time $\tau(x)$, determined from the initial condition, and the end time $s(y)$ given by the free boundary to \eqref{eqn:HJB}, which can be expressed in the following two ways by inverting the Hamiltonian flow:
	\begin{align}
		s\big(\chi_{_{H,J_\psi}}(\tau(x),x)\big)=&\ \tau(x),\nn\\
		s(y)=&\ \tau\big(\chi_{_{H,\psi,s}} (0, y)\big).\nn
	\end{align}

	\begin{theorem}\label{thm:Monge_map}
		Suppose hypotheses {\normalfont\ref{itm:continuity}} - {\normalfont\ref{itm:tCoercive}} hold along with \ref{itm:HCoercive} and either \hyperlink{itm:TC}{\rm \bf [TC]} or \hyperlink{itm:TD}{\rm \bf [TD]}. Under these assumptions the Hamiltonian flow $\chi_H(t_1,t_2,q,\beta)$ is everywhere single-valued and $J_\psi$ is Lipschitz continuous.
		\begin{enumerate}
			
			\item If $\mu$ is absolutely continuous with respect to Lebesgue measure and has disjoint support from $\nu$, then the optimal transport plan is unique and is given by $x \mapsto \chi_{_{H,J_\psi}} (\tau (x), x)$, that is 
			\begin{equation*}
				V(\mu,\nu)=\int _{\R^n} c \big(x, \chi_{_{H,J_\psi}} (\tau (x), x)\big) \mu(dx), 
			\end{equation*}
			where the end time  $\tau(x)$ is almost everywhere the unique time such that $H(\tau,\gamma(\tau),p(\tau))=0$, where $\gamma (0)=x$ and $p(0)=\nabla J_\psi(0, x)$.

			If {\normalfont\ref{itm:Aconvex}}, then the pair $(\rho,\tilde{\rho})$ corresponding to the embedding of the optimal transport plan into the Eulerian formulation, cf.\ Proposition {\normalfont\ref{prop:Eulerian_embedding}}, is a minimizer of the Eulerian formulation.

			\item If $\nu$ is absolutely continuous with respect to Lebesgue measure and has disjoint support from $\mu$, then the optimal transport plan is unique and given by $y\mapsto \chi_{_{H,\psi,s}}(0,y)$. Similarly, in the case of {\normalfont\ref{itm:Aconvex}}, the embedding from Proposition {\normalfont\ref{prop:Eulerian_embedding}} minimizes the Eulerian formulation. In particular, $\tilde{\rho}(dt,dy) = \delta_{s(y)}(dt)\nu(dy)$, where $s$ is the free boundary of \eqref{eqn:HJB}.
		\end{enumerate}

		%In either of these cases, $\rho$ of the Eulerian formulation is also the density of the flow determined by these maps $\rho(t,\cdot,\cdot)\leq {\chi_H}_\#\mu$, and $\tilde{\rho}$ is the stopping distribution.
		% Then $\tilde{\rho}$ has support on $\partial \{(t,\mb{x})\ t> (<) s(\mb{x})\}$.  If $\nu$ is absolutely continuous it suffices to only consider the graph of $s(\mb{x})$. In particular if $\mb{x}$ is a trajectory that ends at $\mb{z}$, then at points of super-differentiability of $\psi$ then end-time must be $s(\mb{z})$.  Similarly, if $\mb{x}$ is a trajectory with costate $\mb{p}$, then there is a unique optimal end-time $\tau$.
	\end{theorem}  

	%REDO this, emphasize how uniqueness proves that this is the unique rho,\tilde{rho} AP.
	\begin{proof}
		We have verified that $\psi$ is Lipschitz continuous in Proposition \ref{prop:MKattainment}.  That $J_\psi$ is Lipschitz continuous follows from hypothesis \ref{itm:HCoercive} and the uniform estimates of \cite{Cannarsa2010Holder}. %A little sketchy because they have fixed end time AP.

		The Hamiltonian equations have a unique solution for the initial (or terminal) value problem using the smoothness assumption of \ref{itm:HCoercive}, thus the Hamiltonian flow $\chi_H$ is well defined for each $(t_1,t_2,q,\beta)$ and continuous.  By Lemma \ref{lem:hitting_time} and almost everywhere differentiability of $\psi$ and $J_\psi$, the stopping times $\tau(x)$ and $s(y)$ are well defined almost everywhere and measurable, so the maps $x \mapsto \chi_{_{H,J_\psi}} (\tau (x), x)$ and $y\mapsto \chi_{_{H,\psi,s}}(0,y)$ exist.

		We now invoke Theorem \ref{thm:Pontryagin} to show that, if $\pi$, $\psi$ and $J_\psi$ are optimal, then the optimal trajectories $\gamma^{x,y}$ solve the Hamiltonian equations with the costate satisfying $p^{x,y}(0)=\nabla J_\psi(0,x)$ and $p^{x,y}(\tau^{x,y})=\nabla \psi(y)$ at points of differentiablity of $J_\psi$ and $\psi$.  Hence, if $\mu$ is absolutely continuous then $\gamma^{x,y},p^{x,y}$ are determined by $\chi_H(0,\cdot,x,\nabla J_\psi(0,x))$ for almost every $x$ and if $\nu$ is absolutely continuous they are determined by $\chi_H(s(y),\cdot,y,\nabla \psi(y))$ for almost every $y$. It follows in either case that $\pi$ is the unique transport plan supported on the graph $(x,\chi_{_{H,J_{\psi}}}(\tau(x),x))$ or $(\chi_{_{H,\psi,s}}(0,y),y)$.

		We now suppose \ref{itm:Aconvex} so that for any Hamiltonian trajectory $(\gamma,p)$ there is a unique control $A(t)$.  By Proposition \ref{prop:Eulerian_embedding}, $\chi_{_{H,J_{\psi}}},\tau,A$ embed into a density process and stopping distribution with cost $V(\mu,\nu)$.  We have proven in Theorem \ref{thm:Eulerian_equality} that $V(\mu,\nu)=W(\mu,\mu)$ so $(\rho,\tilde{\rho})$ are optimal.  Furthermore, $\tilde{\rho}$ has support in the set $(s(y),y)$ and is given by $\delta_{s(y)}(dt)\nu(dy)$.
	\end{proof}

	%In \cite{bernard2005optimal} more tools are developed to resolve the stronger criteria that hold when the initial or target is absoutely continuous.

	\begin{remark}\label{rem:Eulerian_Uniqueness}
		We also conjecture that under the assumptions of Theorem{\normalfont~\ref{thm:Monge_map}} the minimizer of the Eulerian formulation is unique.  This holds if $J_\psi\in C^1(\R^+\times \R^n)$ by uniqueness of the transport equation implied by Theorem {\normalfont\ref{thm:Eulerian_verification}}.  We expect that an argument can be made to handle the case that $J_\psi$ is Lipschitz by using generalized gradients of $J_\psi$, but we leave this to another work.
	\end{remark}

\section{One dimensional examples}\label{sec:examples}
%\marginpar{DONE Feb. 23, 2018}

%	\subsection{Discrete controls with  $\mathbb{A}=\{\pm 1\}$, $f(q,\pm1)=\pm 1$, $K(t,q,\pm1) =g'(t)$.}
%	\marginpar{It will be much better to put an additional example that does not reduce to the usual optimal transport.} 

	To illustrate our analysis in a very simple setting, we now consider a one-dimensional problem where the control set $\mathbb{A}$ consists of two options: travelling to the left or to the right with constant speed. We let the cost be a function of time only, namely as the derivative of a function $g$ with $g(0)=0$ and $g'(t)\geq 0$.  The Hamiltonian is then,
	$$
		H(t,q,p) = \max\{\pm p-g'(t)\}=|p|-g'(t).
	$$
	The optimal control is to travel left if $p<0$ and to the right if $p>0$.  Since the Hamiltonian does not depend on $q$, the costate is constant along trajectories, which will therefore be straight lines with slope 1, illustrated in Figure~\ref{fig-2}.  It is then easy to see that the cost $c(x,y) = g(|y-x|)$, hence the corresponding Monge map $Y(x)$ referred to by Theorem~\ref{thm:Monge_map} can also be obtained from classical results; see e.g., \cite{G-M}. It is however easy to compute in this one dimensional case. 	

	\begin{figure}[!htb]
		\centering
		\begin{subfigure}[b]{0.5\textwidth}
		\centering
        	\includegraphics[width=\linewidth]{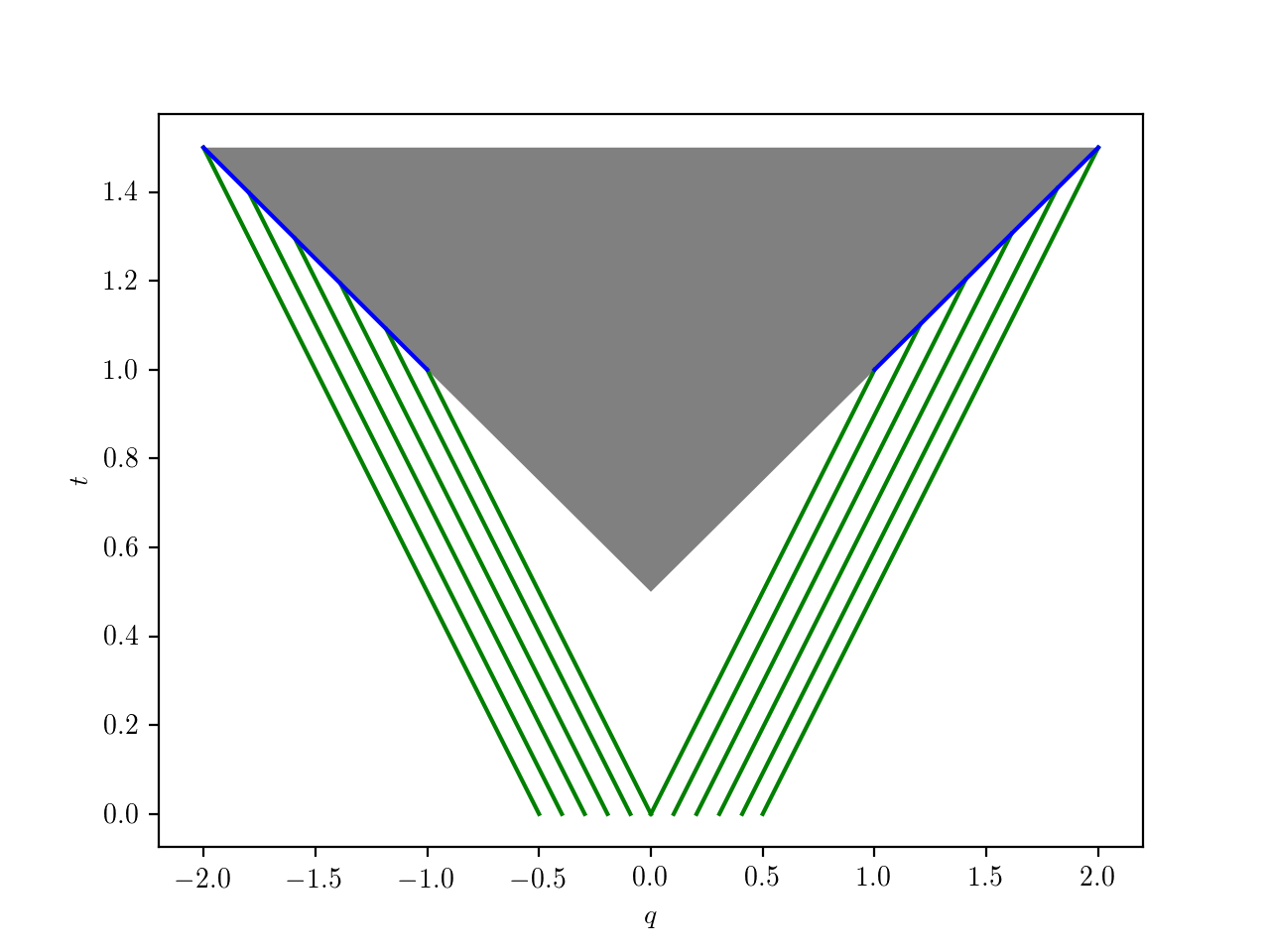}
        	\caption{$g(t)=t^2$}
    	\end{subfigure}%
    	\begin{subfigure}[b]{0.5\textwidth}
		\centering
        	\includegraphics[width=\linewidth]{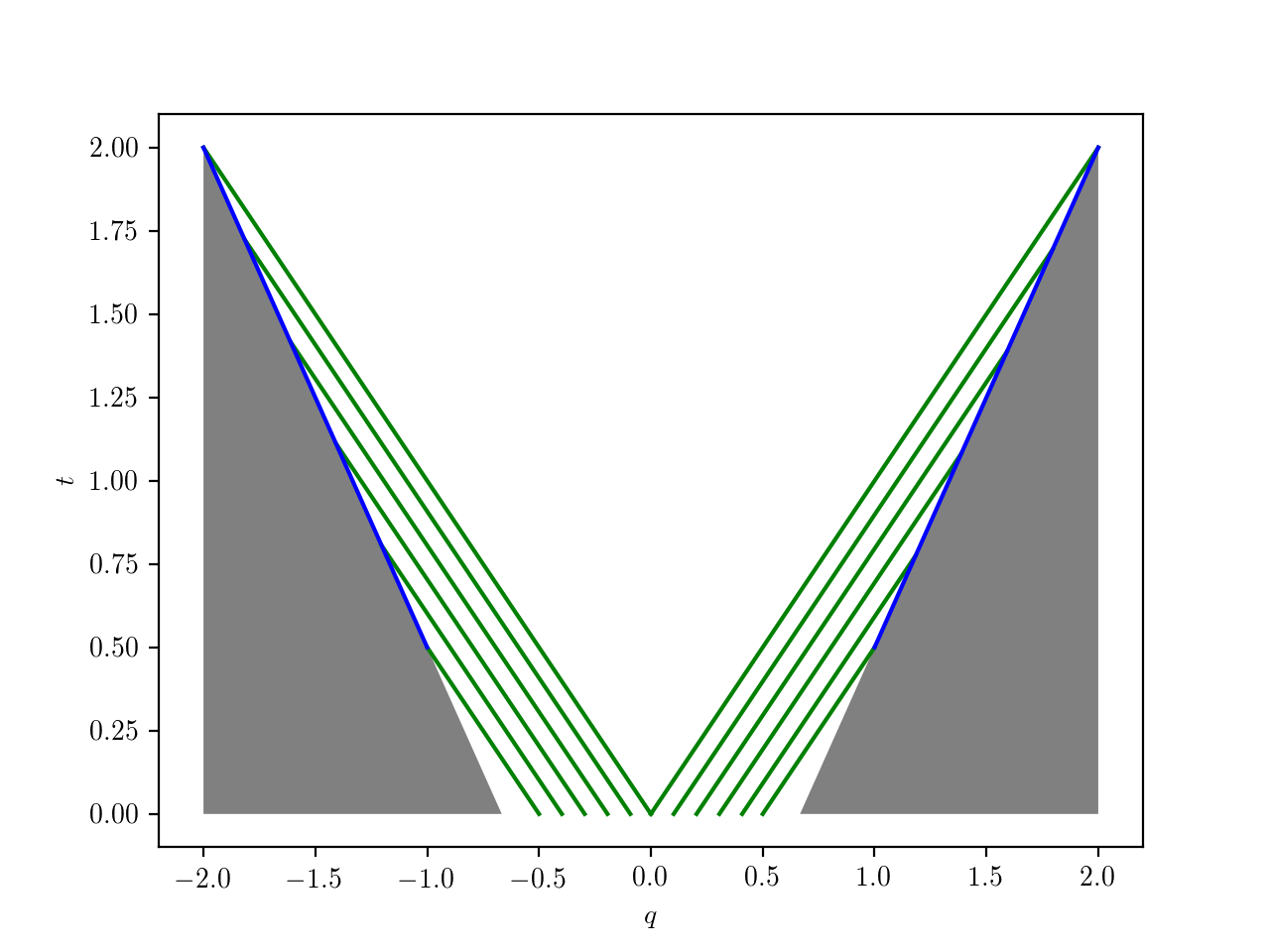}
        \caption{$g(t)=1-e^{-t}$}
    	\end{subfigure}
    	\caption{\label{fig-2}The optimal trajectories in green, which stop upon hitting the free boundary $s$ in blue.  The shaded region shows where $\psi(q)=J_\psi(t,q)$.}
	\end{figure}

	To illustrate how nonsmooth solutions can occur and to explicitly compute the free boundary $s(y)$, and the optimal dual functions $\psi$, $J_\psi$, we assume the initial distribution to have uniform density on a connected interval, 
	 $\frac{d\mu}{dx} = 1_{[-1/2,1/2]}$, 
	while  the target is $\frac{d\nu}{dy}= \frac{1}{2}1_{[-2,-1]}+\frac{1}{2}1_{[1,2]}$.
	
	\begin{enumerate}[label = \Alph*.]
		\item\label{itm:convex-1} If $g$ is convex, that is if the cost $g'$ is monotonically increasing, then the Monge map $Y(x)$ is monotone and can easily be seen to transport the interval $[-1/2,0]$ to $[-2,-1]$.  To satisfy the target constraint,  
		we require that $Y'(x) = \frac{d\mu}{dx}(x)/\frac{d\nu}{dy}(Y(x))=2$. Since $Y(-1/2)=-2$ and $Y(1/2)=2$, then
		$$
			Y(x) = \begin{cases} -1 + 2x, & x<0 \\ 1+2x, & x>0.\end{cases}
		$$
		The travel time is given by $\tau(x)=|Y(x)-x| = 1+|x|$. Inverting $Y$ to express this as the free boundary $s(y)$, we get $$s(y) = \frac{1}{2} + \frac{1}{2}|y|.$$  The costate as a function of the end position, $P(y)$ has the same sign as $y$ and is determined by the transversality 
	 condition \eqref{eqn:transversality}, 
		$$0=H(s(y),y,P(y)) = |P(y)|-g'(s(y)).$$ 
		Thus $P(y) = {\rm sign}(y) g'(\frac{1}{2} + \frac{1}{2}|y|)$. The dual potential $\psi$ satisfies $\nabla \psi(y)=P(y)$, thus can be integrated to obtain
		$$
			\psi(y) = \int_{\pm 1}^y g'\left(\frac{1}{2}\pm \frac{1}{2}q\right)dq +C%= 2\int_{1}^{\frac{1}{2}+ \frac{1}{2}|y|}g'(u)du 
			= 2g\left(\frac{1}{2}+ \frac{1}{2}|y|\right).
		$$
		The function $J_\psi$ is obtained via the formula,
		\begin{align}
			 J_\psi(0,x) &= \psi\big(Y(x)\big) -\int_0^{\tau(x)} g'(s) ds\nn\\
			 &= \psi\big(Y(x)\big) - g\big(|Y(x)-x|\big)\nn\\
			 &=g(1+|x|).\nn
		\end{align}
		At later times, with $t<s(q)$, we have
		$$
			J_\psi(t,q) = J_\psi(0,|q|-t) + \int_0^t g'(s) ds = g(1+|q|-t)+g(t).
		$$
		Note that the inequality $\psi(q)\leq J_\psi(t,q)$ is here an expression of the  convexity of $g$, since
		$
			g\left(\frac{1}{2}+ \frac{1}{2}|q|\right)\leq \frac{1}{2}g(1+|q|-t)+\frac{1}{2}g(t).
		$\\

		\item\label{itm:concave-1} If $g$ is concave, i.e., $g'$ is monotonically decreasing, then the Monge map $Y(x)$ is orientation reversing and satisfies $Y'(x)=-2$. Then $Y(-1/2)=-1$ and $Y(1/2)=1$ and
		$$
			Y(x) = \begin{cases} -2- 2x, & x<0 \\ 2-2x, & x>0.\end{cases}
		$$
		The travel time is $\tau(x) = 2-3|x|$. The free boundary is then $$s(y) = 2-3|1-|y|/2|=-1+3|y|/2.$$  The transversality condition \eqref{eqn:transversality} implies that $\nabla \psi(y) = {\rm sign}(y) g'(-1+3|y|/2)$, which can be integrated to obtain
		$$
			\psi(y) = \int_{\pm 1}^y g'\left(-1\pm \frac{3}{2}q\right)dq +C%= \frac{2}{3}\int_{\frac{1}{2}}^{-1+ \frac{3}{2}|y|}g'(u)du 
			= \frac{2}{3}g\left(-1+ \frac{3}{2}|y|\right).
		$$
		The function $J_\psi$ is then
		$$
		J_\psi(0,x) = \frac{2}{3}g\big(2-3|x|\big)-g\big(2-3|x|\big)=-\frac{1}{3}g(2-3|x|).
						%-J_\psi(0,x) = g\big(2-3|x|\big)-\frac{2}{3}g\big(2-3|x|\big)=\frac{1}{3}g(2-3|x|).
		$$
		At later times, with $t>s(q)$, we have
		$$
			J_\psi(t,q) = J_\psi(0,|q|-t) + g(t) = -\frac{1}{3}g(2-3(|q|-t))+g(t).
		$$
	\end{enumerate}

	\begin{figure}[!htb]
		\centering
		\begin{subfigure}[b]{0.5\textwidth}
		\centering
        	\includegraphics[width=\linewidth]{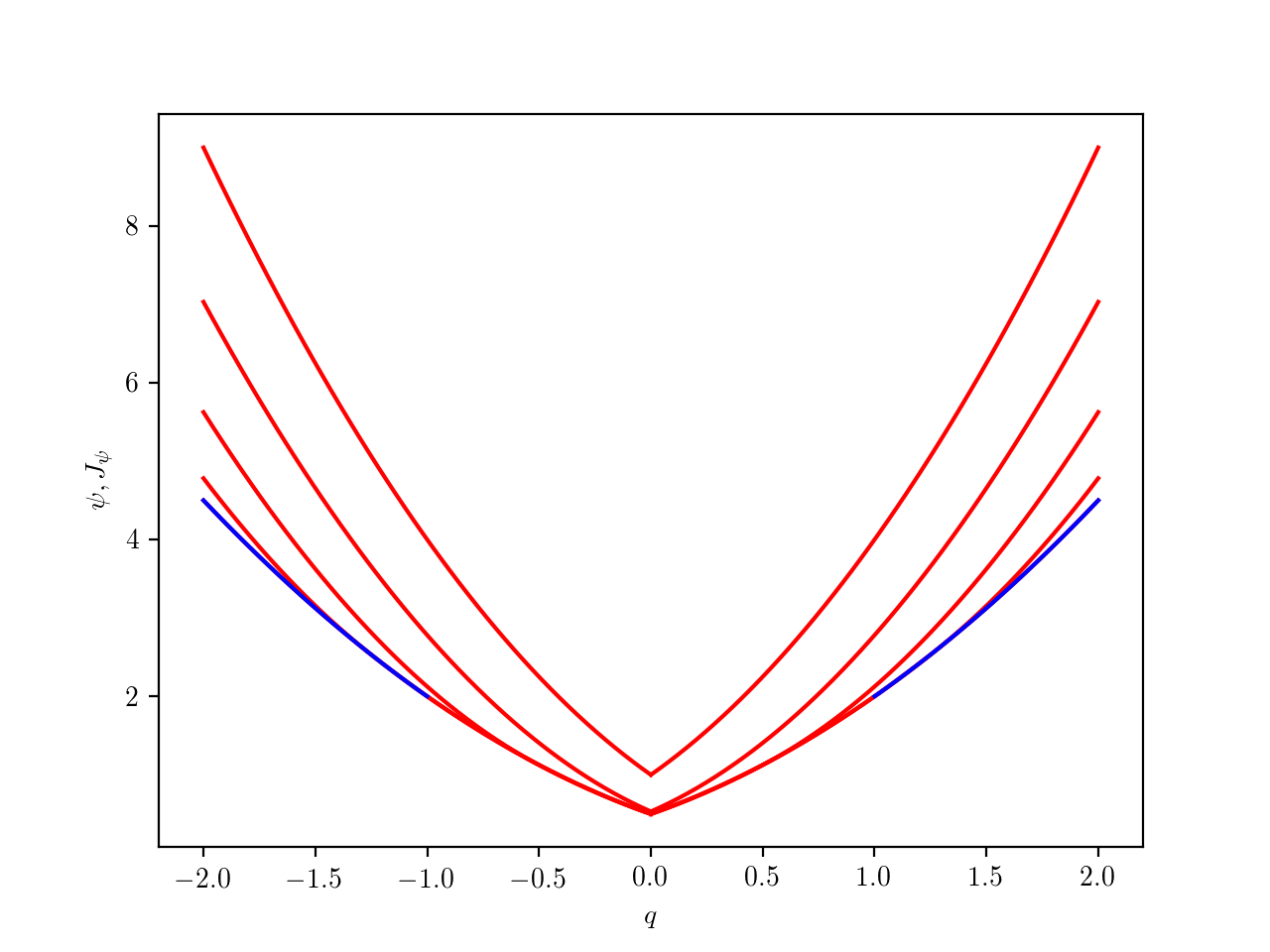}
        	\caption{$g(t)=t^2$}
    	\end{subfigure}%
    	\begin{subfigure}[b]{0.5\textwidth}
		\centering
        	\includegraphics[width=\linewidth]{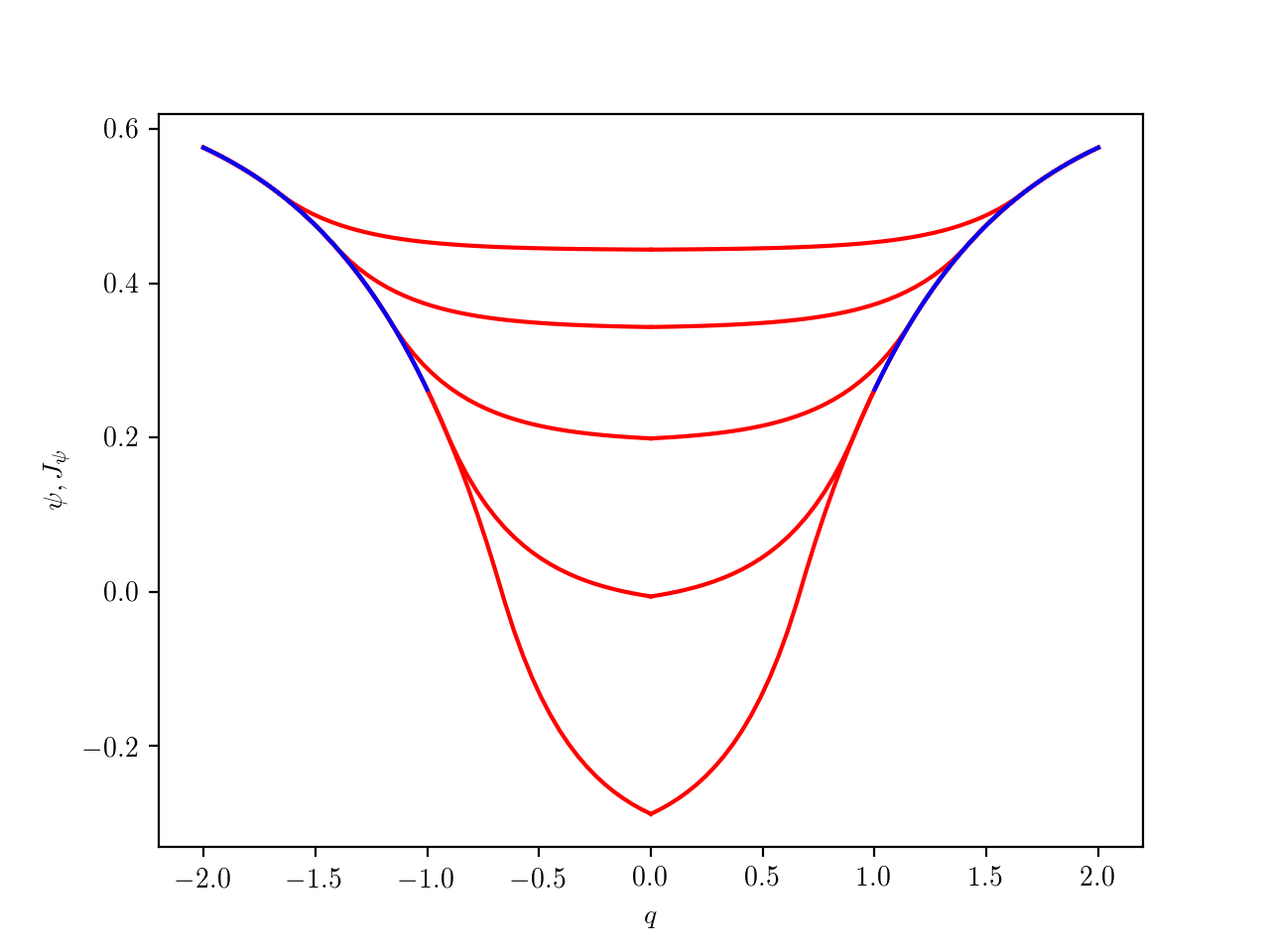}
        \caption{$g(t)=1-e^{-t}$}
    	\end{subfigure}
    	\caption{\label{fig-1}Time slices of $J_\psi$ in red decreasing / increasing in time, and $\psi$ in blue.}
	\end{figure}

	Figure~\ref{fig-1} illustrates how the function $J_\psi$ evolves in time $t$. In the case \ref{itm:convex-1}, the graph $J_\psi$ moves downward as time $t$ increases and mass is dropped when it hits the barrier $\psi$ (hitting time). 
	In the case \ref{itm:concave-1}, the graph $J_\psi$ moves upward as $t$ increases, and mass is dropped when it departs from the barrier $\psi$ (exit time).

\section{Connections with classical optimal transport problems}\label{sec:Connection-OT} %\marginpar{Done Feb. 23. YH.}
	
	We now generalize the above examples to higher dimensions and show a %new 
	  simple 
	connection with the classical optimal transportation problem of Monge (see e.g. \cite{G-M}). %\cite{caffarelli2002constructing}. 
	%{\red In the classical setting, where the cost can be reduced to a convex/concage function of distance (see e.g. \cite{G-M}), our formulation of optimal transport with free end-time, has a simple connection to those with fixed end-times:  }
	We let the control set be the unit sphere, $\mathbb{A}=S^{n-1}\subset \R^n$ and $f(q,A)=A$, we can then view the control problem as minimizing over trajectories with velocity bounded by $1$. Consider now a cost that is only given by a time penalty,
		$$
			K(t,q,A)=g'(t), 
		$$
		where $g$ satisfies $g(0)=0$ and $g'(t)\geq 0$. The optimal trajectory connecting $x$ and $y$ is again the straight line that reaches $y$ at time $\tau^{x, y}=|y-x|$.  Thus, the cost associated to travelling from $x$ to $y$ in the optimal amount of time is
		$$
			c(x,y)=\int_0^{\tau^{x,y}}g'(t)dt = g(|y-x|).
		$$
			In the case that $g$ is strictly convex, $g'(t)$ increases hence satisfies  \hyperlink{itm:TC}{\bf [TC]} 
		and if $g$ is strictly concave then   \hyperlink{itm:TD}
			{\bf [TD]}  
		 is satisfied. The case when $g(t)=t$ is the stationary case  \hyperlink{itm:TS}
			{\bf [TS]}  and corresponds to the classical distance function cost of Monge (see e.g., \cite{caffarelli2002constructing}). %Our results
		Proposition~\ref{prop:dualStructure} then identifies the qualitative change in the solution when the cost changes from being induced by a convex or a concave function $g$. The end time then changes from being a hitting time of a barrier from below to a hitting time from above.  
	
	When $g$ is convex, and as noted in \cite{B-B}, an equivalent dynamic formulation can be posed with fixed endtime $T=1$ and Lagrangian $L(q,v)=g(|v|)$.  The associated Hamilton-Jacobi equation is then
	\begin{align}
		\frac{\partial}{\partial t}I_\psi(t,q)+g^*\big(|\nabla I_\psi(t,q)|\big) =&\ 0,\ 0\leq t<1,\nn\\
		I_\psi(1,y)=&\ \psi(y).\nn
	\end{align}
On the other hand, in our formulation we have the following free boundary problem, % {\red (due to the free end time $s(q)$),}
	\begin{align}
		\frac{\partial}{\partial t}J_\psi(t,q)+|\nabla J_\psi(t,q)|=&\ g'(t),\ 0\leq t<s(q),\nn\\
		J_\psi(t,q)=&\ \psi(q),\ s(q)\leq t,\nn\\
		|\nabla \psi(y)|=&\ g'\big(s(y)\big).\nn
	\end{align}
	The equivalence of these two problems implies that for optimal $\psi$, the solutions satisfy $I_\psi(0,x)=J_\psi(0,x)$.
	Both sets of equations have the characteristic curves given by straight lines, along which $\nabla J_\psi$ or $\nabla I_\psi$ is constant.  The difference is a dynamic rescaling of the length of the curves so that the trajectories in our free end time formulation have constant velocity, while in the classical  formulation they have constant end time.

\bibliography{DeterministicStoppingBib}
\bibliographystyle{plain}

\end{document}